\documentclass[12pt]{article}
\usepackage{amsmath,amsthm,amssymb}
\usepackage{pifont}
\usepackage{circledsteps}
\usepackage{amsfonts}
\usepackage{mathrsfs}
\usepackage{bm}
\usepackage{mathtools}
\usepackage{graphicx}
\usepackage{enumerate}
\usepackage{setspace}
\usepackage{multirow}
\usepackage{booktabs}
\usepackage{comment}

\allowdisplaybreaks[4]

\usepackage[numbers,sort&compress]{natbib}
\usepackage{authblk}
\usepackage{enumitem}
\setlist[enumerate, 1]{label = (\roman*), font = \upshape}

\usepackage[pdfencoding=auto,colorlinks,linkcolor=red,anchorcolor=blue,citecolor=green]{hyperref}
\usepackage{enumitem} 

\usepackage[capitalise,noabbrev]{cleveref}

\usepackage[margin = 1in]{geometry}

\newcommand{\email}[1]{\href{mailto:#1}{\nolinkurl{#1}}}

\usepackage{fouridx}
\newcommand*{\tran}[1]{\,\fourIdx{t\!}{}{}{}{#1}}

\theoremstyle{plain}
\newtheorem{theorem}{Theorem}[section]
\newtheorem{corollary}[theorem]{Corollary}
\newtheorem{lemma}[theorem]{Lemma}
\newtheorem{proposition}[theorem]{Proposition}

\theoremstyle{definition}
\newtheorem{definition}[theorem]{Definition}

\newtheorem{example}{Example}
\newtheorem{remark}{Remark}

\newtheorem{problem}{Problem}

\numberwithin{equation}{section}

\makeatletter

\@addtoreset{equation}{section}

\newcommand{\Rmnum}[1]{\expandafter\@slowromancap\romannumeral #1@}
\makeatother

\def\~#1{\widetilde{#1}}                                        

\DeclareMathOperator{\R}{\mathbb{R}} 

\DeclareMathOperator{\F}{\mathbb{F}}

\DeclareMathOperator{\diag}{\mathrm{diag}}
\DeclareMathOperator{\rank}{\mathrm{rank}}
\DeclareMathOperator{\tr}{\mathrm{tr}}

\title{$P$-polynomial coherent configurations
}

\author{
 Eiichi Bannai\thanks{Faculty of Mathematics, Kyushu University (emeritus), Japan. \texttt{bannai@math.kyushu-u.ac.jp}}, Sho Suda\thanks{Department of Mathematics, National Defense Academy of Japan, 239-8686 Japan \texttt{ssuda@nda.ac.jp}},  Yan Zhu\thanks{School of Mathematics, University of Shanghai for Science and Technology, Shanghai 200093, China  \texttt{zhuyan@usst.edu.cn}}
}

\begin{document}

\maketitle


\begin{abstract}
Suda introduced the notion of a $Q$-polynomial coherent configuration, which provides a natural and important concept. 
Subsequently, Lato introduced a notion of a $P$-polynomial coherent configuration and proved that every such configuration satisfying the definition has at most two fibers. 
Although Lato’s definition is interesting, particularly because it characterizes distance-biregular graphs, we argue that an alternative definition is desirable. 
In this paper, we propose an alternative notion of $P$-polynomial coherent configurations that is naturally aligned with Suda's $Q$-polynomial framework. 
We show that every two-fiber coherent configuration that is $P$-polynomial in Lato’s sense is also $P$-polynomial in our sense, whereas the converse does not hold. 
We further prove that every coherent configuration of type $(2,2;3)$, $(3,2;3)$ or $(3,3;3)$ is $P$-polynomial in our sense.
In addition, we present three families of $P$-polynomial coherent configurations with an arbitrary number of fibers: those arising from tight Euclidean
$t$-designs in $\mathbb R^2$, the Terwilliger algebra of $H(n,2)$, and the set of all subspaces of $\mathbb F_q^n$. 
Finally, we give an equivalent condition for the cross-block intersection matrices to be
tridiagonal and verify that all three families satisfy this condition.
\end{abstract}
\section{Introduction}

Coherent configurations were introduced independently by Weisfeiler and Leman \cite{WL-2018} in their study of the graph isomorphism problem and by Higman \cite{Higman-1975} as an algebraic framework for investigating permutation groups and their representations. They generalize several important combinatorial and algebraic structures, including association schemes, orbitals of permutation groups, and coherent algebras. In particular, association schemes are precisely the homogeneous coherent configurations, namely those with a single fiber.
Among the known examples, the quasi-symmetric designs investigated by Goethals and Seidel \cite{GS-1970} give rise to coherent configurations of type $(2,2;3)$, and Lison{\v{e}}k \cite{Lisonek-1997} constructed a maximal two-distance set in $\R^8$ that also determines a coherent configuration of this type. Bannai and Bannai \cite{BB-2010} proved that, under suitable conditions, certain Euclidean $2e$-designs supported by two concentric spheres admit the structure of coherent configurations. Bannai, Bannai, Tanaka, and Zhu \cite{BBTZ-2022} further showed that, under a mild condition, a tight relative $2e$-design on two shells of the binary Hamming association scheme also induces a coherent configuration.

In 2022, Suda \cite{Suda-2022} introduced the concept of a $Q$-polynomial coherent configuration, extending the theory of $Q$-polynomial association schemes to coherent configurations with multiple fibers. 
In that paper, he established several equivalent conditions of the $Q$-polynomial property analogous to those in the theory of association schemes. 
In particular, the $Q$-polynomial property is characterized in terms of the tridiagonal structure of Krein matrices and the polynomial dependence of the dual eigenvalues.
Based on this definition, he provided several families of important examples, including tensor products of $Q$-polynomial association schemes, the disjoint union of spherical designs, the tight Euclidean $2e$-designs on two concentric spheres, as well as the tight relative $2e$-designs on two shells in the binary Hamming association scheme. 
The results motivate the subsequent study of the conditions under which Euclidean embeddings of $Q$-polynomial coherent configurations with two fibers give rise to Euclidean $t$-designs (see \cite{JZW-2026}).

Motivated by the work of Suda, Lato \cite{Lato-2025} introduced a concept of $P$-polynomial coherent configurations as the dual notion of $Q$-polynomial coherent configurations. 
Under this definition, every $P$-polynomial coherent configuration can have at most two fibers. 
In particular, every such configuration arises either from a distance-regular graph or from a distance-biregular graph.
The examples considered are distance-biregular graphs, quasi-symmetric designs, and strongly regular designs.
Furthermore, Fernández, Ihringer, Lato, and Munemasa \cite{FILM-2025} gave several new constructions of distance-biregular graphs. 
The authors introduce an infinite family arising from hyperovals and a new sporadic example belonging to a generalization of a construction due to Delorme.
They also establish new non-existence criteria for distance-biregular graphs ruling out the existence of a distance-biregular graph on $225+60$ vertices

In this paper, we propose an alternative notion of $P$-polynomial coherent configurations that is naturally aligned with Suda's definition of a $Q$-polynomial property.
Unlike Lato’s definition, which uses a common generator for the whole block row, our definition allows a separate generator for each block. 
We first show that, in the two-fiber case, Lato's
$P$-polynomial property implies ours. 
We then prove that every coherent configuration of type $(2,2;3)$, $(3,2;3)$, or $(3,3;3)$ is $P$-polynomial in our sense.
We also establish the $P$-polynomial property for three
families with an arbitrary number of fibers, namely, coherent configurations $\mathcal C_1$
arising from tight Euclidean $t$-designs in $\mathbb R^2$, the coherent configuration $\mathcal C_2$ associated with the Terwilliger
algebra of $H(n,2)$, and the coherent configuration $\mathcal C_3$ associatedn with  the set of all subspaces of $\mathbb F_q^n$. 
Furthermore, we give an equivalent condition under which
the cross-block intersection matrices are tridiagonal, and verify that all three families satisfy this condition.

The paper is organized as follows. In \cref{sect:experiment}, we recall coherent configurations and $Q$-polynomial 
property introduced by Suda, and review the orthogonal polynomials used in the examples presented later. In \cref{sect:pray}, we introduce our definition and compare it with Lato's definition.
In \cref{sect:scandal}, we give an equivalent condition for the tridiagonality of cross-block intersection matrices. 
\cref{sect:revenge} studies the three families with an arbitrary number of fibers.

\section{Preliminaries and orthogonal polynomials}\label{sect:experiment}
In this section, we recall Higman's definition of a coherent configuration \cite{Higman-1975} and Suda's definition of a $Q$-polynomial coherent configuration \cite{Suda-2022}. 
We also review the Gegenbauer, Hahn, and $q$-Hahn polynomials, together with their
dual families. 
These orthogonal polynomials play a fundamental role in
the three examples presented in \cref{sect:revenge}. 
For further background on orthogonal polynomials, we refer the reader to \cite{LSK-2010}.

\subsection{Coherent configuration and $Q$-polynomial property}
\begin{definition}[{\cite{Higman-1975}}]\label{def:operation}
Let $X$ be a non-empty finite set and $\mathcal{R}=\{R_i \; |\; i \in I\}$ be a set of non-empty subsets of $X \times X$.  
The pair $\mathcal{C}=(X, \mathcal{R})$ is called a coherent configuration if it satisfies the following conditions.    
        \begin{enumerate}
            \item \label{itm:def1(1)} $\mathcal{R}$ is a partition of $X \times X$. 
            \item  If $R_i \in \mathcal{R}$ and $R_i \cap \mathrm{diag}(X \times X) \neq \emptyset$, where $\mathrm{diag}(X \times X)=\{(x, x) | x \in X\}$, then $R_i \subseteq \mathrm{diag}(X \times X)$. 
            \item \label{itm:def1(3)} $R_i \in \mathcal{R}$ implies $\tran{R_i}=\{(y,x) \; | \; (x,y) \in R_i\} \in \mathcal{R}$. 
            \item For any $i,j,h\in I$, the number $|\{z\in X \mid (x,z)\in R_i,(z,y)\in R_j\}|$ is independent of the choice of $(x,y)\in R_h$. 
        \end{enumerate}
\end{definition}
Let $A_i$ be the adjacency matrix of the graph $(X,R_i)$ for $i\in I$. 
We define the coherent algebra $\mathcal{A}$ of the coherent configuration $\mathcal{C}$ as the subalgebra of $\mathrm{Mat}_{|X|}(\mathbb{C})$ generated by $\{A_i \mid i\in I\}$ over $\mathbb{C}$. 
There uniquely exists a subset $\Omega$ in $I$ such that $\mathrm{diag}(X\times X)=\sqcup_{i\in \Omega}R_i$ by \cref{def:operation}\ref{itm:def1(1)} and \ref{itm:def1(3)}.

For a subset $R$ of $X\times X$, define the projection of $R$ as follows:
\begin{align*}
\mathrm{pr}_1(R)&=\{x\in X\mid (x,y)\in R \text{ for some }y\in X\},\\
\mathrm{pr}_2(R)&=\{y\in X\mid (x,y)\in R \text{ for some }x\in X\}.
\end{align*}
We obtain the standard partition $\{X_i \mid i\in \Omega\}$ of $X$ where $X_i=\mathrm{pr}_1(R_i)=\mathrm{pr}_2(R_i)$ for $i\in\Omega$. 
We call $X_i$ a {\it{fiber}} of the coherent configuration $\mathcal{C}$.  
The following property of binary relations of coherent configurations was shown in \cite{Higman-1975}.
\begin{lemma}[{\cite{Higman-1975}}]\label{lem:relation}
For any $i\in I$, there exist $j,h\in \Omega$ such that $\mathrm{pr}_1(R_i)=X_j$, $\mathrm{pr}_2(R_i)=X_h$.
\end{lemma}
For $i,j\in\Omega$, define $I^{(i,j)}=\{\ell\in I\mid R_\ell\subset X_i\times X_j\}$. 
By \cref{lem:relation}
, the collection $\{I^{(i,j)}\mid i,j\in\Omega\}$ forms a partition of $I$. 
Set $r_{i,j}=|I^{(i,j)}|-\delta_{i,j}$. The matrix $\left(|I^{(i,j)}|\right)_{i,j\in\Omega}$
is called the \emph{type} of the coherent configuration $\mathcal C$. For instance, a quasi-symmetric design naturally gives rise to a coherent configuration of type $(2,2;3)$.
Let $\varepsilon_{i,j}=1-\delta_{i,j}$. Renumber the elements of $I^{(i,j)}$ as $R^{(i,j)}_{\varepsilon_{i,j}},\ldots,R^{(i,j)}_{r_{i,j}}$ where $R^{(i,i)}_0=\operatorname{diag}(X_i\times X_i)$ and $\tran{R^{(i,j)}_\ell}=R^{(j,i)}_\ell$.
Let $A^{(i,j)}_\ell$ denote the adjacency matrix of $R^{(i,j)}_\ell$, and let
\[
\mathcal A^{(i,j)}=
\operatorname{span}_{\mathbb C}
\left\{A^{(i,j)}_\ell \mid \varepsilon_{i,j}\le \ell\le r_{i,j}\right\}.
\]
Then $\mathcal A^{(i,j)}\mathcal A^{(j,h)}
\subseteq
\mathcal A^{(i,h)}$ and the intersection numbers
$p^{(i,j,h)}_{\ell,m,n}$
are defined by
\begin{equation}\label{eqn:banner}
A^{(i,j)}_\ell A^{(j,h)}_m=\sum_{n=\varepsilon_{i,h}}^{r_{i,h}} p^{(i,j,h)}_{\ell,m,n} A^{(i,h)}_n.
\end{equation}
Define $k^{(i,j)}_\ell=p^{(i,j,i)}_{\ell,\ell,0}$. 
Then $k^{(i,j)}_\ell=\left| \left\{y\in X_j\mid (x,y)\in R^{(i,j)}_\ell\right\} \right|$ for any $x\in X_i$. 
We call $k^{(i,j)}_\ell$ the valency of $R^{(i,j)}_\ell$.

Let $\widetilde r_{i,j}=r_{i,j}-\varepsilon_{i,j}$.
In this paper, we consider coherent configurations $\mathcal C$ such that each fiber $\mathcal C^i=(X_i,I^{(i,i)})$is a symmetric association scheme for every $i\in\Omega$, and such that $\mathcal A$ admits a basis
\[
\left\{E_h^{(i,j)}\mid i,j\in\Omega,\, 0\le h\le\widetilde r_{i,j}\right\}
\]
satisfying the following conditions.
\begin{enumerate}
\item[(B1)] For any $i,j\in\Omega$, $E^{(i,j)}_0
=\frac{1}{\sqrt{|X_i||X_j|}} J_{|X_i|,|X_j|}$, 
where $J_{p,q}$ denotes the $p\times q$ all-ones matrix.

\item[(B2)] For any $i,j\in\Omega$, $\left\{E_h^{(i,j)}\mid 0\le h\le\widetilde r_{i,j}\right\}$ is a basis of $\mathcal A^{(i,j)}$.

\item[(B3)] For any $i,j\in\Omega$ and
$\ell\in\{0,1,\ldots,\widetilde r_{i,j}\}$, $\tran{E^{(i,j)}_\ell}=E^{(j,i)}_\ell$.
\item[(B4)] For any
$i,j,i',j'\in\Omega$,
$\ell\in\{0,\ldots,\widetilde r_{i,j}\}$,
and
$\ell'\in\{0,\ldots,\widetilde r_{i',j'}\}$,
\begin{equation}\label{eqn:maze}
 E^{(i,j)}_\ell E^{(i',j')}_{\ell'} =\delta_{j,i'}\delta_{\ell,\ell'} E^{(i,j')}_\ell.   
\end{equation}

\end{enumerate}

Set $m_h^{(i,j)}=\rank{E_h^{(i,j)}}$.
Since $\mathcal A^{(i,j)}$ is closed under the Hadamard product $\circ$, the Krein parameters
$q^{(i,j)}_{\ell,m,n}$
are defined by
\[
E^{(i,j)}_\ell\circ E^{(i,j)}_m=\frac{1}{\sqrt{|X_i||X_j|}}\sum_{n=0}^{\widetilde r_{i,j}}q^{(i,j)}_{\ell,m,n}
E^{(i,j)}_n.
\]
For $i,j\in\Omega$, since $\left\{A_\ell^{(i,j)}\mid \varepsilon_{i,j}\leq \ell\leq r_{i,j}\right\}$ and $\left\{E_h^{(i,j)}\mid 0\leq h\leq \widetilde{r}_{i,j}\right\}$ are bases of $\mathcal{A}^{(i,j)}$, there exist change-of-bases matrices  $P^{(i,j)}=(P_{\ell}^{(i,j)}(h))_{\substack{0\leq h \leq \widetilde{r}_{i,j} \\ \varepsilon_{i,j}\leq \ell \leq r_{i,j} }}$ and
$Q^{(i,j)}=(Q_h^{(i,j)}(\ell))_{\substack{\varepsilon_{i,j}\leq \ell\leq r_{i,j}\\ 0\leq h\leq \widetilde{r}_{i,j}}}$ such that  
\begin{align*}
(A_{\varepsilon_{i,j}}^{(i,j)},\ldots,A_{r_{i,j}}^{(i,j)})&=(E_0^{(i,j)},\ldots,E_{\widetilde{r}_{i,j}}^{(i,j)})P^{(i,j)}, \\
(E_0^{(i,j)},\ldots,E_{\widetilde{r}_{i,j}}^{(i,j)})&=\frac{1}{\sqrt{|X_i||X_j|}}(A_{\varepsilon_{i,j}}^{(i,j)},\ldots,A_{r_{i,j}}^{(i,j)})Q^{(i,j)}, 
\end{align*}
equivalently, 
\begin{equation}\label{eqn:election}
A_\ell^{(i,j)}=\sum_{h=0}^{\widetilde{r}_{i,j}} P_\ell^{(i,j)}(h)E_h^{(i,j)},\quad E_h^{(i,j)}=\frac{1}{\sqrt{|X_i| |X_j|}}\sum_{\ell=\varepsilon_{i,j}}^{r_{i,j}} Q_h^{(i,j)}(\ell)A_\ell^{(i,j)}.
\end{equation}

\begin{proposition}[{\cite{Suda-2022}}]\label{prop:heart}
The following relations hold:
\begin{enumerate}
 \item \label{item:prop(1)}  $P_{\ell}^{(i, j)}(0)=\sqrt{\frac{\left|X_i\right|}{\left|X_j\right|}} k_{\ell}^{(i, j)}$,
\item \label{item:prop(2)} $Q_0^{(i, j)}(\ell)=1$, 
\item \label{item:prop(3)} $\frac{Q_h^{(i, j)}(\ell)}{\sqrt{\left|X_j\right|} m_h^{(i, j)}}=\frac{P_{\ell}^{(i, j)}(h)}{\sqrt{\left|X_i\right|} k_{\ell}^{(i, j)}}$,
\item $P^{(i,j)}Q^{(i,j)}=\sqrt{|X_i||X_j|} I$.
\end{enumerate}
\end{proposition}
Suda \cite{Suda-2022} introduced the concept of $Q$-polynomial coherent configuration.
\begin{definition}[\cite{Suda-2022}] \label{def:embark} 
A coherent configuration is said to be \emph{$Q$-polynomial} if, for every
$i,j\in\Omega$, there exists a set of polynomials $\left\{v_h^{(i, j)}(x) \mid 0 \leq h \leq \widetilde{r}_{i, j}\right\}$ 
such that
\[
\sqrt{|X_i||X_j|}\,E_h^{(i,j)}=v_h^{(i,j)} \left(\sqrt{|X_i||X_j|} E_1^{(i,j)}\right),
\]
where the polynomial is evaluated under the Hadamard product, and $\deg v_h^{(i,j)}=h.$
\end{definition}

\subsection{Gegenbauer polynomials}
Gegenbauer polynomials $\left\{G_k(u)\right\}_{k=0}^{\infty}$ are defined by the following recurrence relation:
\[\begin{aligned}
& G_0(u)=1, \quad G_1(u)=d u, \\
& \frac{k+1}{d+2 k} G_{k+1}(u)=u G_k(u)-\frac{d+k-3}{d+2 k-4} G_{k-1}(u) .
\end{aligned}\]
If $d=2$, then $G_k(u)=2 T_k(u)$ for $k\geq 1$, where $T_k(u)$ is the Chebyshev polynomial of the first kind.
It is known that $\cos k \theta=T_k(\cos \theta)$
and the following composition property holds 
\begin{equation}\label{eqn:trainer}
T_{k_1}(T_{k_2}(u))=T_{k_2}(T_{k_1}(u))=T_{k_1k_2}(u).
\end{equation}
Then we have the following result.
\begin{lemma}\label{lem:curtain}
Define polynomials $f_\ell(u)$ of degree $\ell$  by the following recurrence  relation:
\begin{equation}\label{eqn:expansion}
f_0(u)=1, \quad f_1(u)=u,\quad  f_{\ell}(u)=(u+1)f_{\ell-1}(u)-f_{\ell-2}(u), 
 \quad \ell \geq 2.    
\end{equation}
Then $\frac{\cos(2 \ell-1)h \theta}{\cos h \theta}=f_{\ell-1}\left(\frac{\cos 3h \theta}{\cos h \theta}\right)$ for any fixed $\theta$.
\end{lemma}
\begin{proof}
Let $y=\cos h \theta$.
By \cref{eqn:trainer}, we have $\frac{\cos (2 \ell-1) h \theta}{\cos h \theta}=\frac{T_{2\ell-1}(y)}{y}$
and  $\frac{\cos 3 h \theta}{\cos h \theta}=\frac{T_3(y)}{y}=4 y^2-3$.
Denote $u=\frac{\cos 3 h \theta}{\cos h \theta}$, i.e. $u=4 y^2-3$.
Note that $\frac{T_{2\ell-1}(y)}{y}$ is a polynomial and an even function as well. 
Therefore, there exist polynomials $g_\ell(z)$ and $f_\ell(z)$ of degree $\ell$ such that
\[\frac{T_{2 \ell-1}(y)}{y}=g_{\ell-1}\left(y^2\right)=f_{\ell-1}\left(4 y^2-3\right).\]
Using the recurrence relation of Chebyshhev polynomial $T_n(y) =2 y T_{n-1}(y)-T_{n-2}(y)$, we obtain
\[\begin{aligned}
T_{2\ell+1}(y)  =2 y T_{2\ell}(y)-T_{2\ell-1}(y)  &=2 y \cdot\left(2 y \cdot T_{2\ell-1}(y)-T_{2\ell-2}(y)\right)-T_{2\ell-1}(y) \\
&=\left(4 y^2-2\right) T_{2\ell-1}(y)-T_{2\ell-3}(y) .
\end{aligned}\]
Then $\frac{T_{2\ell+1}(y)}{y}=\left(4 y^2-2\right) \frac{T_{2\ell-1}(y)}{y}-\frac{T_{2\ell-3}(y)}{y}$.
 Recall that $u=4 y^2-3$.
It follows from $\frac{T_{2\ell-1}(y)}{y}=f_{\ell-1}\left(4 y^2-3\right)$ that $f_{\ell}(u)=(u+1)f_{\ell-1}(u)-f_{\ell-2}(u)$. 
\end{proof}
\subsection{Hahn and dual Hahn polynomials}
The shifted factorial (also called the Pochhammer symbol) is defined by
\[
(a)_0:=1,\qquad (a)_k:=\prod_{i=1}^{k}(a+i-1), \qquad k=1,2,3,\ldots.
\]
The hypergeometric function ${}_rF_s$ is defined by the series
\[
{}_rF_s\!\left(
\begin{matrix}
a_1,\ldots,a_r\\
b_1,\ldots,b_s
\end{matrix}
;z
\right)
:=
\sum_{k=0}^{\infty}
\frac{(a_1,\ldots,a_r)_k}
{(b_1,\ldots,b_s)_k}
\frac{z^k}{k!},
\]
where $(a_1,\ldots,a_r)_k:=(a_1)_k\cdots(a_r)_k$.
For positive integers $n,i,j$ with $n-i\geq j$,
 define
\begin{equation}\label{eqn:anniversary}
Q_k(u;-(n-i)-1,-i-1,j)={}_3 F_2\Big(
\begin{array}{cc}
-k,-u,-n+k-1\\
-n+i,-j
\end{array};1\Big).    
\end{equation}
to be the {\it{Hahn polynomial}} of degree $k$ with respect to $u$.
The {\it{dual Hahn polynomial}} is defined by 
\[R_k(\lambda(u);-(n-i)-1,-i-1,j)={}_3 F_2\Big(
\begin{array}{cc}
-u,-k,-n+u-1\\
-n+i,-j
\end{array};1\Big),\]
where $\lambda(u)=u(u-n-1)$.
The relation between the Hahn polynomial and the dual Hahn polynomial is given by the following identity:
\begin{equation}\label{eqn:stress}
Q_k(u;-(n-i)-1,-i-1,j)=R_u(\lambda(k);-(n-i)-1,-i-1,j)
\end{equation}


\subsection{$q$-Hahn and dual $q$-Hahn polynomials}
For $q\neq0$ and $q\neq1$, define a $q$-analogue of the Pochhammer symbol $(a)_k$:
\[(a;q)_0:=1,\quad
(a;q)_k:=\prod_{i=1}^{k}(1-aq^{\,i-1}),
\qquad
k=1,2,3,\ldots.\]
For nonnegative integer $n$, the $q$-binomial coefficient is defined by
\[{n \brack k}_q:=\frac{(q;q)_n}{(q;q)_k(q;q)_{n-k}},
\qquad
k=0,1,\ldots,n.
\]
In the following, we denote ${n \brack k}_q$ as ${n \brack k}$ for simplicity.
The $q$-hypergeometric function ${}_r\phi_s$ is defined by the series
\[{}_r\phi_s\!\left(
\begin{matrix}
a_1,\ldots,a_r\\
b_1,\ldots,b_s
\end{matrix}
;q,z
\right)
:=
\sum_{k=0}^{\infty}
\frac{(a_1,\ldots,a_r;q)_k}
{(b_1,\ldots,b_s;q)_k}
(-1)^{(1+s-r)k}
q^{(1+s-r)\binom{k}{2}}
\frac{z^k}{(q;q)_k},\]
where $(a_1,\ldots,a_r;q)_k:=(a_1;q)_k\cdots(a_r;q)_k.$
The {\it{$q$-Hahn polynomial}} is defined by 
\[Q_k(u;q^{-(n-i)-1},q^{-i-1},j|q)=  {}_3 \phi_2\Big(
\begin{array}{cc}
q^{-u},q^{-k},q^{-n+k-1}\\
q^{-n+i},q^{-j}
\end{array};q,q\Big).\]
The {\it{dual $q$-Hahn polynomial}} is defined by 
\[R_k(\mu(u);q^{-(n-i)-1},q^{-i-1},j|q)=  {}_3 \phi_2\Big(
\begin{array}{cc}
q^{-k},q^{-u},q^{-n+u-1}\\
q^{-n+i},q^{-j}
\end{array};q,q\Big),\]
where $\mu(u)=q^{-u}+q^{-n+u-1}$. 
It is clear that 
\[Q_k(u;q^{-(n-i)-1},q^{-i-1},j|q)=R_u(\mu(k);q^{-(n-i)-1},q^{-i-1},j|q).\]

\section{$P$-polynomial coherent configuration}\label{sect:pray}
In this section, we recall Lato's definition of a $P$-polynomial coherent configuration and introduce our new definition.  
We prove that every two-fibre coherent configuration that is $P$-polynomial in Lato's sense is also $P$-polynomial in our sense, while the converse
does not hold in general.
More specifically, every coherent configuration type $(2,2;3),$ $(3,2;3)$ or $(3,3;3)$ is automatically $P$-polynomial in our sense,\textbf{} whereas the specific one arising from \cite[Example~5.6]{BBTZ-2022} does not satisfy
Lato's definition.  
\subsection{Lato's definition}

Lato~\cite{Lato-2025} introduced the notion of a $P$-polynomial
coherent configuration and proved that such a configuration has at
most two fibers. 
Moreover, its adjacency algebra is associated with either a
distance-regular graph or a distance-biregular graph.
For $1\leq i\leq p$, set $r_i=\sum_{j=1}^p r_{i,j}.$
For each fixed $i$, let
\[
\sigma_i:
\left\{
(j,\ell):
1\leq j\leq p,\ 
\varepsilon_{i,j}\leq\ell\leq r_{i,j}
\right\}
\longrightarrow
\{0,1,\ldots,r_i\}
\]
be a bijection, and define
\[
M_{\sigma_i(j,\ell)}^i=A_\ell^{(i,j)}.
\]
Denote $P_\ell^{(i,j)}(h):=P_{\sigma_i(j,\ell)}^i(h)$ for the corresponding eigenvalue.
\begin{definition}[{\cite{Lato-2025}}]\label{def:formal}
A coherent configuration is said to be $P$-polynomial if, for every
$1\leq i\leq p$, there exist an ordering $\sigma_i$ as above and
polynomials $\left\{v_k^i(x)\right\}_{k=0}^{r_i}$ with $\deg v_k^i=k$ such that
\[
P_\ell^{(i,j)}(h)
=
v_{\sigma_i(j,\ell)}^i\left(P_1^i(h)\right)
\]
for all $1\leq j\leq p$,
$\varepsilon_{i,j}\leq\ell\leq r_{i,j}$, and
$0\leq h\leq \widetilde r_{i,i}$.
\end{definition}
Let $\mathcal C$ be a $P$-polynomial coherent configuration. Then the number of fibers is at most two. 
Suppose that $\mathcal C$ has two fibers. In this case, the
generators are $M_1^1=A_1^{(1,2)}$ and $M_1^2=A_1^{(2,1)}$, and the corresponding global adjacency matrix is
\[
A_1=
\begin{pmatrix}
0&A_1^{(1,2)}\\
A_1^{(2,1)}&0
\end{pmatrix}.
\]
With the  ordering, for
$\{i,j\}=\{1,2\}$, we have
\[
M_{2\ell}^i=A_\ell^{(i,i)},
\qquad
M_{2\ell+1}^i=A_{\ell+1}^{(i,j)}.
\]
hus, the even-indexed matrices are supported on the diagonal blocks, whereas the odd-indexed matrices are supported on the off-diagonal blocks. 
Under the correspondence with the associated bipartite graph, these matrices represent the even- and odd-distance relations, respectively. 
Consequently, there exist polynomials $u_k^i$ of
degree $k$ such that
\begin{equation}\label{eqn:rebellion}
P_\ell^{(i,i)}(h) = u_{2\ell}^i\left(P_1^i(h)\right), \qquad
P_{\ell+1}^{(i,j)}(h)=u_{2\ell+1}^i\left(P_1^i(h)\right).    
\end{equation}
Since even- and odd-indexed relation matrices are supported on the diagonal and off-diagonal blocks, respectively, we may choose the polynomials as follows
\begin{equation}\label{eqn:spell}
u_{2\ell}^i(x)=f_\ell(x^2),
    \qquad
u_{2\ell+1}^i(x)=x g_\ell(x^2),    
\end{equation}
for some polynomials with $\deg f_\ell=\deg g_\ell=\ell.$
In \cite{Lato-2025}, it was shown that $P$-polynomial coherent configurations with two fibers are equivalently characterized as distance-biregular graphs.
\begin{definition}[{\cite{Lato-2025}}]\label{def:volunteer}
 Let $G=(X_1\cup X_2,E)$ be a connected bipartite graph with bipartition $(X_1, X_2)$, and for
$u \in V(G)$ let $\Gamma_i(u)=\{v\in V(G)\mid \partial(u,v)=i\}$.
The graph $G$ is called \emph{distance-biregular} if, for every $u\in V(G)$ and every $v \in\Gamma_i(u)$, the numbers
\[
 c_i(u)=|\Gamma(v)\cap\Gamma_{i-1}(u)|,
 \qquad
 b_i(u)=|\Gamma(v)\cap\Gamma_{i+1}(u)|
\]
are independent of the choice of $v$, and depend on $u$ only through
whether $u\in X_1$ or $u\in X_2$.   
\end{definition}
\subsection{Our definition}

For an association scheme, the primitive idempotents determine
normalized spherical functions, each normalized by its value at the
identity relation $A_0$, namely, $P_0(h)=1$.
An off-diagonal block of a coherent configuration does not contain an identity relation. 
To obtain an analogous normalization, we choose the
distinguished relation $A_{\varepsilon_{i,j}}^{(i,j)}$ as a
reference and normalize the eigenvalues in each row of the block
eigenmatrix $P^{(i,j)}$ so that its first column is the all-ones
vector.
\begin{definition}\label{def:responsible}
For each \(i,j\in\Omega\), suppose  $P_{\varepsilon_{i,j}}^{(i,j)}(h) \neq 0 \  (0\leq h\leq\widetilde r_{i,j})$,
and define
\[
\overline{P_\ell^{(i,j)}}(h)
=
\frac{P_\ell^{(i,j)}(h)}
     {P_{\varepsilon_{i,j}}^{(i,j)}(h)}.
\]
A coherent configuration $\mathcal C=(X,\mathcal R)$ is said to be $P$-polynomial if, for every \(i,j\in\Omega\), there
exist an ordering of $\left\{
A_\ell^{(i,j)}
\right\}_{\varepsilon_{i,j}\leq\ell\leq r_{i,j}}$ and polynomials $\left\{v_\ell^{(i,j)}(x)\right\}_{\varepsilon_{i,j}\leq\ell\leq r_{i,j}}$ with 
$\deg v_\ell^{(i,j)}=\ell-\varepsilon_{i,j}$ such that
\[
\overline{P_\ell^{(i,j)}}(h)
=
v_\ell^{(i,j)}
\left(
\overline{P_{1+\varepsilon_{i,j}}^{(i,j)}}(h)
\right)
\]
for every admissible $0 \leq h \leq \widetilde{r}_{i,j}$.
\end{definition}

\begin{remark}
The normalization in \cref{def:responsible} is analogous to the standard normalization of zonal spherical functions in the theory of Gelfand pairs, where each spherical function is normalized by its value at the identity element.
\end{remark}

\subsection{Comparison of the two definitions in the two-fibre case}

In the one-fiber case, both definitions reduce
to the usual $P$-polynomial property for association schemes and
are therefore equivalent. 
For the two-fiber case, we prove that \cref{def:responsible} is strictly weaker than \cref{def:formal}.

\begin{proposition}\label{prop:remedy}
Every two-fiber coherent configuration that is $P$-polynomial in the sense of \cref{def:formal} is also $P$-polynomial in the sense of \cref{def:responsible}.
\end{proposition}

\begin{proof}
Fix $\{i,j\}=\{1,2\}$, and set $x_h=P_1^{(i,j)}(h)$.
By \cref{eqn:rebellion,eqn:spell}, there exist polynomials $f_\ell^i$ and $g_\ell^i$ of degree $\ell$, such that
\begin{equation}\label{eqn:sleep}
P_\ell^{(i,i)}(h)=f_\ell^i(x_h^2),
\qquad
P_{\ell+1}^{(i,j)}(h)
=
x_hg_\ell^i(x_h^2).    
\end{equation}
We claim that $x_h \neq 0$ for every $h$. Otherwise, \cref{eqn:sleep} implies $P_{\ell+1}^{(i,j)}(h)=0$ for all $\ell$.
Then the $h$-th row of $P^{(i,j)}$ is zero, which contradicts the orthogonality relation 
\[
P^{(i,j)}Q^{(i,j)}
=
\sqrt{|X_i||X_j|}\,I.
\]
Hence the normalizations in \cref{def:responsible} are well-defined. \\
For the diagonal block, \cref{eqn:sleep} gives 
$\overline{P_\ell^{(i,i)}}(h)=f_\ell^i(x_h^2).$
Since $f_1^i$ has degree one, $x_h^2$ is a polynomial of degree one in $\overline{P_1^{(i,i)}}(h)=f_1^i(x_h^2)$.
Then there exists a polynomial $v_\ell^{(i,i)}$ of
degree $\ell$ such that
\[
\overline{P_\ell^{(i,i)}}(h)
=
v_\ell^{(i,i)}
\left(
\overline{P_1^{(i,i)}}(h)
\right).
\]
For the off-diagonal block, \cref{eqn:sleep} gives $\overline{P_{\ell+1}^{(i,j)}}(h)
=
\frac{P_{\ell+1}^{(i,j)}(h)}
     {P_1^{(i,j)}(h)}
=
g_\ell^i(x_h^2)$.
Similarly, $x_h^2$ is a polynomial of degree one in $\overline{P_2^{(i,j)}}(h)=g_1^i(x_h^2)$.
Therefore, there exists a polynomial $v_{\ell+1}^{(i,j)}$ of degree $\ell$ such that
\[
\overline{P_{\ell+1}^{(i,j)}}(h)
=
v_{\ell+1}^{(i,j)}
\left(
\overline{P_2^{(i,j)}}(h)
\right).
\]
Hence the coherent configuration satisfies \cref{def:responsible}.
\end{proof}

Next, we consider coherent configurations of small type. We show that those of types $(2,2;3)$, $(3,2;3)$, and $(3,3;3)$ are automatically $P$-polynomial in the sense of \cref{def:responsible}, and conclude with a
counterexample of type $(3,2;3)$ that is not $P$-polynomial in the sense of \cref{def:formal}.

\begin{proposition}\label{prop:wreck}
Every coherent configuration of type $(2,2;3)$, $(3,2;3)$, or $(3,3;3)$ is $P$-polynomial in the sense of
\cref{def:responsible}.
\end{proposition}
\begin{proof}
For each of these types, every diagonal block is a symmetric
association scheme of class one or two, and hence is
$P$-polynomial. It remains to consider the off-diagonal blocks.
\smallskip

\noindent
\textbf{Case I: Types $(2,2;3)$ and $(3,2;3)$.}\\
For $i\ne j$, the block $\mathcal{A}^{(i,j)}$ has two adjacency
matrices, say $A_1^{(i,j)}$ and $A_2^{(i,j)}$. After interchanging
them if necessary, we may assume that
\[
 P_1^{(i,j)}(h)\ne0,
 \qquad h=0,1.
\]
Indeed, $P_\ell^{(i,j)}(0)>0$ for both $\ell=1,2$, and if both
entries in the second row of $P^{(i,j)}$ were zero, then
$P^{(i,j)}$ would be singular. Thus the normalization is
well-defined.
Since $P^{(i,j)}$ is nonsingular, so is
$\overline{P^{(i,j)}}$. 
As its first column is the all-ones vector, its second column is nonconstant. 
We  take $v_1^{(i,j)}(x)=1$ and $v_2^{(i,j)}(x)=x$ with  $\deg v_\ell^{(i,j)}=\ell-1$ so that
\[
 \overline{P_{\ell}^{(i,j)}}(h)
 =
 v_{\ell}^{(i,j)}
 \left(\overline{P_{2}^{(i,j)}}(h)\right),
 \qquad \ell=1,2,\quad h=0,1.
\]
Hence every off-diagonal block is $P$-polynomial.

\smallskip
\noindent
{\bf{Case II: Type $(3,3;3)$}}.\\
For $i\ne j$, it follows from type $(3,3;3)$ that its first eigenmatrix is
\[
P^{(i,j)}
=
\left(P_\ell^{(i,j)}(h)\right)_{
\substack{0\leq h\leq2\\1\leq\ell\leq3}}.
\]
The orthogonality relation gives
\begin{equation}\label{eqn:ladder}
\sum_{h=0}^{2}
m_h^{(i,j)}
P_\ell^{(i,j)}(h)P_{\ell'}^{(i,j)}(h)
=
|X_i|k_\ell^{(i,j)}\delta_{\ell,\ell'},
\end{equation}
with
\begin{equation}\label{eqn:heel}
P_\ell^{(i,j)}(0)
=
\sqrt{\frac{|X_i|}{|X_j|}}\,k_\ell^{(i,j)}>0.
\end{equation}
{\bf{Step 1:}} We show that $P^{(i,j)}$ has a column with no zero entries.
Otherwise, by \cref{eqn:heel}, every column would have a zero in row $1$ or $2$. 
Two columns cannot have zeros in different rows, since their inner product in \cref{eqn:ladder} would be $m_0^{(i,j)}
P_\ell^{(i,j)}(0)P_{\ell'}^{(i,j)}(0)>0.$
Thus, every column would have a zero entry in the same row, contradicting the nonsingularity of $P^{(i,j)}$. 
After relabelling, we may assume that
\[
P_1^{(i,j)}(h) \ne 0,\qquad h=0,1,2.
\]
Define
\[
f_\ell(h)
=
\overline{P_\ell^{(i,j)}}(h)
=
\frac{P_\ell^{(i,j)}(h)}{P_1^{(i,j)}(h)},
\qquad \ell=1,2,3.
\]
Then $f_1=1$, and \cref{eqn:ladder} shows that
$f_1,f_2,f_3$ are orthogonal with respect to the positive weights $w_h=m_h^{(i,j)}
\bigl(P_1^{(i,j)}(h)\bigr)^2.$\\
\noindent
{\bf{Step 2:}} We show that one of $f_2$ and $f_3$ has pairwise distinct values at $h=0,1,2$.
If $f_2$ does not, since it is not constant, then we may relabel so that $f_2(0)=f_2(1)=a\ne b=f_2(2).$
The orthogonality of $f_3$ to $f_1$ and $f_2$  gives
\[
\sum_{h=0}^{2}w_hf_3(h)=0,
\qquad
a\bigl(w_0f_3(0)+w_1f_3(1)\bigr)
+bw_2f_3(2)=0.
\]
These equations imply $f_3(2)=0$ and
$w_0f_3(0)+w_1f_3(1)=0.$
Since $f_3$ is nonzero and the weights are positive, its three
values are pairwise distinct. Thus, after interchanging $f_2$
and $f_3$ if necessary, we may assume that
$f_2(0),f_2(1),f_2(2)$ are mutually distinct.\\
\noindent
{\bf{Step 3:}} By Lagrange interpolation, there exists a polynomial
$v_3^{(i,j)}$ of degree at most $2$ satisfying
\[
f_3(h)=v_3^{(i,j)}\bigl(f_2(h)\bigr),
\qquad h=0,1,2.
\]
Moreover, the degree cannot be $1$. 
Otherwise $f_3\in\operatorname{span}\{f_1,f_2\}$, this contradicts the orthogonality. 
Hence $\deg v_3^{(i,j)}=2$. 
Taking $v_1^{(i,j)}(x)=1$ and $v_2^{(i,j)}(x)=x$,
 gives the polynomial relations required in
\cref{def:responsible}.
\end{proof}

Tight relative $2e$-designs in $H(n,2)$ supported on two shells were studied in \cite{BBTZ-2022}.
We immediately obtain the following result. 
\begin{corollary}\label{cor:related}
The coherent configuration induced on a tight relative $4$-design in $H(n,2)$ supported on two shells is $P$-polynomial in the sense of \cref{def:responsible}.
\end{corollary}

\begin{remark}
In \cite{BBTZ-2022}, it was shown that a tight relative $2e$-design in $H(n,2)$ supported on two shells is $Q$-polynomial.
However, \cref{cor:related} is specific to $e=2$. 
For $e>2$, the induced coherent configuration has type $(e+1,e;e+1)$. 
Its diagonal blocks are association schemes with $e$
classes, which are not automatically $P$-polynomial.
\end{remark}

We conclude this section with a tight relative $4$-design whose induced coherent configuration of type $(3,2;3)$ does not satisfy \cref{def:formal}.

\begin{example}\label{ex:decorative}
Let $(Y,\mathbf{1})$ be the tight relative $4$-design in
$H(22,2)$ obtained from the Witt $4$-$(23,7,1)$ design, where $Y=Y_6\sqcup Y_7.$
Then the coherent configuration induced on $Y$ is not
$P$-polynomial in the sense of \cref{def:formal}.
\end{example}

\begin{proof}
This design arises from the Witt $4$-$(23,7,1)$ design, where
$Y_6$ and $Y_7$ are its derived and residual designs,
respectively. The two relations on $Y_6\times Y_7$ are determined by $|x\cap y|=1$ and $|x\cap y|=3$, where $x\in Y_6$ and $y\in Y_7$. For each $t\in\{1,3\}$, let
$G_t$ be the bipartite graph with bipartition $Y_6\sqcup Y_7$
defined by
\[
x\sim_{G_t}y
\quad\Longleftrightarrow\quad
|x\cap y|=t.
\]
Fix $x\in Y_6$. For $x'\in Y_6\setminus\{x\}$, we have
$|x\cap x'|\in\{0,2\}$. A direct calculation gives
\[
|N_{G_1}(x)\cap N_{G_1}(x')|
=
\begin{cases}
36 & \text{ if } |x\cap x'|=0,\\
56 & \text{ if } |x\cap x'|=2,
\end{cases}
\]
and
\[
|N_{G_3}(x)\cap N_{G_3}(x')|
=
\begin{cases}
20 & \text{ if } |x\cap x'|=0,\\
40 & \text{ if } |x\cap x'|=2.
\end{cases}
\]
Since all four common-neighbor numbers above are positive, every $x'\in Y_6\setminus\{x\}$ lies at distance two from $x$ in both $G_1$ and $G_3$. 
However, the number of common neighbors of $x$
and $x'$ depends on whether $|x\cap x'|=0$ or $2$. 
Hence neither $G_1$ nor $G_3$ is distance-biregular by \cref{def:volunteer}. 
It follows that the coherent configuration induced on $Y$ is not $P$-polynomial in the sense of \cref{def:formal}.
\end{proof}
\begin{remark}
\cref{def:formal} is based on the global adjacency matrix $A_1$, whereas \cref{def:responsible} is imposed separately on each block. 
The normalization, generator, and ordering may vary from one block to another, so the blockwise polynomial relations need not arise from a common global generating matrix. 
By \cref{prop:wreck}, the coherent configuration in \cref{ex:decorative} is $P$-polynomial in the sense of \cref{def:responsible}, but not in the sense of \cref{def:formal}.
Accordingly, in the next section, we focus on the tridiagonal structure of the block-wise intersection matrices, which provides a local analogue of the classical three-term recurrence.
\end{remark}

Coherent configurations of type $(3,2;3)$, $(3,3;3)$ are known as strongly regular designs in \cite{HIGMAN1988411}, strongly regular designs of the second kind \cite{MR1345694}, respectively. 
We present examples from coding theory. 
\begin{example}
    The dual of the extended ternary Golay code of length $12$ is a $5$-design with degree $3$ in $H(12,3)$. Its derived designs, say $X_1,X_2,X_3$, are $4$-designs in $H(11,2)$ with $s_{ij}=2$ for any $i,j\in\{1,2\}$, see \cite[Proposition~3.2]{gavrilyuk2025extremal}. Then by \cite[Theorem~5,5]{Suda-2022} and Proposition~\ref{prop:wreck}, the union of $X_1 \bigcup X_2 \bigcup X_3$ with the binary relations determined by the Hamming distances is a $P$- and $Q$-polynomial coherent configuration of type $(3,2,2;3,2;3)$. 
    Deleting one of any fiber yields a $P$- and $Q$-polynomial coherent configuration of type $(3,2;3)$. 
\end{example}

\begin{example}
    The doubly shortened code of the extended Golay code of length $22$ is a $5$-design with degree $3$ in $H(22,2)$. Its derived designs, say $X_1,X_2$, are $4$-designs in $H(21,2)$ with $s_{ii}=2$ for any $i\in\{1,2\}$ and $s_{ij}=3$ for any distinct $i,j\in\{1,2\}$, see \cite[Proposition~3.2]{gavrilyuk2025extremal}. Then by \cite[Theorem~5,5]{Suda-2022} and Proposition~\ref{prop:wreck}, the union of $X_1 \bigcup X_2$ with the binary relations determined by the Hamming distances is a $P$- and $Q$-polynomial coherent configuration of type $(3,3;3)$. 
\end{example}

\section{Tridiagonality of cross-block intersection matrices} \label{sect:scandal}
For distinct $i,j\in\Omega$, define the cross-block intersection matrix by
\[
B_1^{(i,i,j)}
=
\left(p_{1,\ell,\ell'}^{(i,i,j)}\right)_{\varepsilon_{i,j}\leq \ell,\ell' \leq r_{i,j}}.
\]
In this section, we will give an equivalent condition for the tridiagonality of $B_1^{(i,i,j)}$.
\begin{proposition}\label{prop:lump}
Let $\mathcal C$ be a $P$-polynomial coherent configuration in the sense of \cref{def:responsible}. 
Fix distinct $i,j\in\Omega$ such that $\widetilde r_{i,j}\geq1$. 
Then the followings are equivalent:
\begin{enumerate}
\item For every $0\leq h\leq\widetilde r_{i,j}$, there exist constants
$\alpha_{i,j},\beta_{i,j}\in\mathbb R$, with
$\alpha_{i,j}\neq0$, such that
\[
\overline{P_1^{(i,i)}}(h)
=
\alpha_{i,j}\overline{P_2^{(i,j)}}(h)
+\beta_{i,j}.
\]
\item The cross-block intersection matrix $B_1^{(i,i,j)}$ is tridiagonal.
\end{enumerate}
Moreover, if these conditions hold, then $\alpha_{i,j}
=
p_{1,1,2}^{(i,i,j)}$ and $\beta_{i,j}
=
p_{1,1,1}^{(i,i,j)}$.
\end{proposition}

\begin{proof}
(i) $\Rightarrow$ (ii): 
For $0\leq h\leq\widetilde r_{i,j}$, set $\theta_h=\overline{P_2^{(i,j)}}(h)$.
By \cref{def:responsible}, there exist polynomials
$v_\ell^{(i,j)}$ with
$\deg v_\ell^{(i,j)}=\ell-1$, such that
\begin{equation}\label{eqn:penny}
 P_\ell^{(i,j)}(h)=P_1^{(i,j)}(h)
v_\ell^{(i,j)}(\theta_h),
\qquad
1\leq\ell\leq r_{i,j}.   
\end{equation}
Substituting \cref{eqn:penny} into the orthogonal relation \cref{eqn:ladder}, we obtain
\[
\sum_{h=0}^{\widetilde r_{i,j}}
m_h^{(i,j)}
\bigl(P_1^{(i,j)}(h)\bigr)^2
v_\ell^{(i,j)}(\theta_h)
v_{\ell'}^{(i,j)}(\theta_h)
=
|X_i|k_\ell^{(i,j)}\delta_{\ell,\ell'}.
\]
Thus $\{v_\ell^{(i,j)}\}_{\ell=1}^{r_{i,j}}$
is a system of orthogonal polynomials on the points $\{\theta_h\}_{h=0}^{\widetilde r_{i,j}}$. 
Hence multiplication by the linear polynomial $\alpha_{i,j} x+\beta_{i,j}$ is
tridiagonal in this basis. 
Equivalently, for every point $\theta_h$,
\begin{equation}\label{eqn:neighbour}
(\alpha_{i,j}\theta_h+\beta_{i,j})v_\ell^{(i,j)}(\theta_h)
=
c_\ell v_{\ell-1}^{(i,j)}(\theta_h)
+a_\ell v_\ell^{(i,j)}(\theta_h)
+b_\ell v_{\ell+1}^{(i,j)}(\theta_h).    
\end{equation}
Since $\overline{P_1^{(i,i)}}(h)=P_1^{(i,i)}(h)$ and $\overline{P_1^{(i,i)}}(h)=\alpha_{i,j}\overline{P_2^{(i,j)}}(h)+\beta_{i,j}$,
\cref{eqn:penny,eqn:neighbour} imply
\begin{equation}\label{eqn:enhance}
P_1^{(i,i)}(h)P_\ell^{(i,j)}(h)=
P_1^{(i,j)}(h)
\bigl( c_\ell v_{\ell-1}^{(i,j)}(\theta_h)+a_\ell v_\ell^{(i,j)}(\theta_h)+b_\ell v_{\ell+1}^{(i,j)}(\theta_h)
\bigr).    
\end{equation}
On the other hand, from \cref{eqn:banner}, we obtain
\begin{equation} \label{eqn:future}
P_1^{(i,i)}(h)P_\ell^{(i,j)}(h)=P_1^{(i,j)}(h) \sum_{\ell'=1}^{r_{i,j}} p_{1,\ell,\ell'}^{(i,i,j)} v_{\ell'}^{(i,j)}(\theta_h).    
\end{equation}
Since $P_1^{(i,j)}(h)\neq0$ for every admissible $h$, comparison of \cref{eqn:enhance,eqn:future}, together with the
linear independence implied by the orthogonality, yields
\[
p_{1,\ell,\ell'}^{(i,i,j)}=0
\qquad\text{whenever}\qquad
|\ell-\ell'|>1.
\]
Therefore $B_1^{(i,i,j)}$ is tridiagonal.\\
(ii) $\Rightarrow$ (i): Suppose that $B_1^{(i,i,j)}$ is tridiagonal, then we have
\begin{equation}\label{eqn:volcano}
A_1^{(i,i)}A_1^{(i,j)}=p_{1,1,1}^{(i,i,j)}A_1^{(i,j)}+p_{1,1,2}^{(i,i,j)}A_2^{(i,j)}.    
\end{equation}
Comparing the coefficients of $E_h^{(i,j)}$ on both sides of \cref{eqn:volcano}, we obtain
\[
P_1^{(i,i)}(h)P_1^{(i,j)}(h)=p_{1,1,1}^{(i,i,j)}P_1^{(i,j)}(h)+p_{1,1,2}^{(i,i,j)}P_2^{(i,j)}(h).
\]
Dividing by $P_1^{(i,j)}(h)\neq0$ gives
\begin{equation}\label{eqn:wilderness}
 \overline{P_1^{(i,i)}}(h)=p_{1,1,2}^{(i,i,j)}
\overline{P_2^{(i,j)}}(h)+p_{1,1,1}^{(i,i,j)}.   
\end{equation}
It remains to show that
$p_{1,1,2}^{(i,i,j)}\neq0$.
Otherwise, \cref{eqn:wilderness} would imply that $P_1^{(i,i)}(h)$ is constant for $0\leq h\leq\widetilde r_{i,j}$.
However, the $P$-polynomial property of the diagonal block imply that the eigenvalues $P_1^{(i,i)}(0),\ldots,
P_1^{(i,i)}(\widetilde r_{i,i})$ are mutually distinct.
This is a contradiction. 
Therefore $p_{1,1,2}^{(i,i,j)}\neq0.$
The proof is complete by taking $\alpha_{i,j}=p_{1,1,2}^{(i,i,j)}$ and $\beta_{i,j}=p_{1,1,1}^{(i,i,j)}$.
\end{proof}

\section{$P$- and $Q$-polynomial coherent configurations with more than two fibers} \label{sect:revenge}

In this section, we present three families of $P$- and $Q$-polynomial coherent configurations in the sense of \cref{def:responsible} and \cref{def:embark}, each with arbitrarily many fibers. 
They arise from tight Euclidean $t$-designs in $\R^2$, the coherent configuration associated
with the Terwilliger algebra of $H(n,2)$, and the configuration of all subspaces of $\F_q^n$.

\subsection{Coherent configuration from tight Euclidean $t$-designs in $\mathbb R^2$}
Bajnok \cite{Bajnok-2006} constructed tight Euclidean $t$-designs in $\mathbb R^2$ for any $t$.
Bannai-Bannai-Hirao-Sawa \cite{BBHS-2010} prove that every tight $t$-design of $\R^2$
or a tight $t$-design on $p$ concentric spheres is similar to one of the
examples given by Bajnok if it is supported by at most
$\left\lfloor \frac{t}{4} \right\rfloor +1$ concentric spheres.
\begin{lemma}[{\cite[Theorem 9]{Bajnok-2006}}] \label{lem:project}
 Let $1 \leq p \leq\left\lfloor\frac{t+5}{4}\right\rfloor$ and $N=t+3-2 p$. Let $r_1<r_2<\cdots<r_p$ be  distinct positive real numbers. For integers $k$ and $i$, define the point
$$
b_{i, k}=r_i \cdot \left(\cos \frac{(2k +i)\pi}{N}, \sin \frac{(2k +i)\pi}{N} \right),
$$
and let
\[X_i=\left\{b_{i,k} \mid 1 \leq k \leq N\right\} \text{~ for ~} 1 \leq i \leq p.\]
Denote $X=X_1\cup X_2\cup \cdots \cup X_p$. 
Furthermore, define the weight function $w: X \rightarrow \mathbb{R}^{+}$ by
\[w\left(b_{i,k}\right)= 
\begin{cases}
\frac{1}{r_1^N} & \text { if } \quad i=1, \\ 
(-1)^i \frac{1}{r_i^N} \prod_{\substack{2\leq l\leq p\\l\neq i}} \frac{r_1^2-r_l^2}{r_i^2-r_l^2} & \text { if } \quad 2 \leq i \leq p.
\end{cases}\]
Then $(X, w)$ is a tight Euclidean $t$-design in $\mathbb R^2$. 
 \end{lemma} 
For $1\leq i,j\leq p$, set
\[
 I(X_i,X_j)=
 \left\{\frac{\langle x,y\rangle}{r_i r_j}
 \mid x\in X_i,\ y\in X_j\right\}\setminus\{1\},
 \qquad s_{i,j}=|I(X_i,X_j)|,
\]
and put $\xi_{i,j}=0$ if $i\equiv j\pmod 2$ and
$\xi_{i,j}=1$ otherwise.  efine
\[
 r_{i,j}=
 \begin{cases}
 s_{i,j}+1,&i\ne j\text{ and }i\equiv j\pmod 2,\\
 s_{i,j},&\text{otherwise}.
 \end{cases}
\]
More precisely, for $\varepsilon_{i,j}\leq\ell\leq r_{i,j}$, put
\[
 \alpha_\ell^{(i,j)}
 =\cos\frac{\bigl(2(\ell-\varepsilon_{i,j})+\xi_{i,j}\bigr)\pi}{N}.
\]
For $x\in X_i$, $y\in X_j$, define the relation
$R_{\ell}^{(i,j)}$ by
\[
 (x,y)\in R_{\ell}^{(i,j)}
 \quad\text{if}\quad
 \frac{\langle x,y\rangle}{r_i r_j}=\alpha_{\ell}^{(i,j)}.
\]

 \begin{lemma}[{\cite[Theorem 2.6]{MR2677683}, \cite[Theorem 5.10]{Suda-2022}}] \label{lem:finish} 
Let $Y_i$ be a spherical $t_i$-design on $S^{d-1}$ for $i \in\{1,2, \ldots, p\}$. 
Assume that $p\geq 3$ and $Y_i \cap Y_j=\emptyset$ for distinct integers $i$ and $j$.  
If one of the following holds depending on the choice of $i, j, h \in\{1,2, \ldots, p\}$
\begin{enumerate}[label=(\roman*)]
\item $s_{i, j}+s_{j, h}-2 \leq t_j$,
\item $i=j=h$, $2 s_{i, i}-3=t_i$, and $Y_i=-Y_i$,  
\end{enumerate}
then $\left(\bigcup_{i=1}^p Y_i,\left\{R_{\ell}^{(i, j)} \mid 1 \leq i, j \leq p, \varepsilon_{i, j} \leq \ell \leq s_{i, j}\right\}\right)$ is a $Q$-polynomial coherent configuration. 
More precisely, define the primitive idempotents $E_h^{(i, j)}$ as follows
\begin{equation}
\begin{aligned}\label{eqn:dealer}
  E_h^{(i, j)}= 
  \begin{cases}
  \frac{1}{\sqrt{\left|X_i\right|\left|X_j\right|}} \sum_{\ell=\varepsilon_{i, j}}^{s_{i, j}} G_h\left(\alpha_\ell^{(i, j)}\right) A_\ell^{(i, j)} & \text { if } \quad h \leq s_{i, j}-1, \\ 
  \Delta_{X_i}-\frac{1}{\left|X_i\right|} \sum_{\ell=\varepsilon_{i, j}}^{s_{i, i}}\left(\sum_{\nu=0}^{s_{i, i}-1}  G_\nu\left(\alpha_\ell^{(i,i)}\right)\right) A_\ell^{(i, i)} & \text { if } \quad i=j \text { and } h=s_{i, i}.
  \end{cases}
\end{aligned}  
\end{equation}  
 \end{lemma}
 \begin{remark}\label{rmk:criminal}
    For tight Euclidean designs in \cref{lem:project}, the following statements hold.
      \begin{enumerate}
      \item $|X_i|=N=t-2 p+3$ and $s_{i,i}=\lfloor \frac{N}{2} \rfloor$ for all $i\in \{1,2,\ldots,p\}$.
     \item  $\frac{1}{r_i} X_i \cup \frac{1}{r_j} X_j $ is exactly  $\frac{1}{r_i} X_i$ if $i \equiv j \pmod 2$ and $\frac{1}{r_i} X_i \cup \frac{1}{r_j} X_j \cong \frac{1}{r_1} X_1 \cup \frac{1}{r_2} X_2$ if $ i\not\equiv j \pmod 2$. 

 \end{enumerate}
 \end{remark}
Let $\mathcal C_1$ be the coherent configuration obtained from a tight Euclidean $t$-design in $\R^2$. 
We obtain the following result immediately from \cref{lem:finish}.
 \begin{corollary}
 Suppose that either $i=j$ or $i\not\equiv j\pmod 2$.
 For $\varepsilon_{i,j}\leq\ell\leq s_{i,j}$ and
 $0\leq h\leq s_{i,j}-\varepsilon_{i,j}$, the $(\ell,h)$-entry
 of the second eigenmatrix $Q^{(i,j)}$ of $\mathcal C_1$ is given by

 \begin{equation}\label{eqn:concrete}
  Q^{(i,j)}_{h}(\ell)=
 \begin{cases}
 G_{h}(\alpha_{\ell}^{(i,j)}) & \text { if}\quad 0\leq h \leq s_{i,j}-1, \\ 
N \cdot \delta_{\ell,0}-\sum_{\nu=0}^{s_{i,i}-1} G_\nu\left(\alpha_{\ell}^{(i,i)}\right) & \text { if}\quad   i=j \text { and } h =s_{i,i}.
\end{cases}   
 \end{equation}
 \end{corollary}
\begin{remark}
Assume that $i\ne j$ and $i\equiv j\pmod 2$. Define a bijection
$\phi_{i,j}:X_i\longrightarrow X_j$ by
\[
\phi_{i,j}(b_{i,k})=b_{j,k+(i-j)/2},
\]
where the second index is interpreted modulo $N$. Then $\frac{1}{r_j}\phi_{i,j}(b_{i,k})
=
\frac{1}{r_i}b_{i,k}.$
Since $\varepsilon_{i,j}=1$, $\varepsilon_{i,i}=0$, and
$\xi_{i,j}=\xi_{i,i}=0$, we have
\[
\alpha_{\ell+1}^{(i,j)}
=
\cos\frac{2\ell\pi}{N}
=
\alpha_\ell^{(i,i)}.
\]
Consequently, the bijection $\phi_{i,j}$ identifies
$R_\ell^{(i,i)}$ with $R_{\ell+1}^{(i,j)}$. 
The associated change-of-basis matrices satisfy
\begin{equation}\label{eqn:junior}
P_{\ell+1}^{(i,j)}(h)=P_\ell^{(i,i)}(h),
\qquad
Q_h^{(i,j)}(\ell+1)=Q_h^{(i,i)}(\ell),
\qquad
0\leq \ell,h\leq s_{i,i}.
\end{equation}
\end{remark}
It follows from \cref{eqn:concrete,eqn:junior} that  
\[m_h^{(i,j)}=m_h^{(i,i)}=\tr(E_h^{(i,i)})=
\begin{cases}
G_h(1) &  \text { if } \quad  0\leq h\leq s_{i,j}-1,\\
 N-\sum_{\nu=0}^{s_{i, i}-1} G_\nu(1)&  \text { if } \quad
 i\equiv j\pmod 2 \text{ and } h= s_{i,i}.
\end{cases}
\]
Since $|X_i|=N$ for any $1\leq i\leq p$, by \cref{prop:heart}, the $(h,\ell)$-entry of the first eigenmatrix $P^{(i,j)}$ is obtained
\begin{equation}\label{eqn:version}
 P_\ell^{(i,j)}(h)=Q_h^{(i,j)}(\ell) \frac{k_\ell^{(i,j)}}{m_h^{(i,j)}}.
\end{equation}
Denote
$K=\diag(k^{(i,j)}_\ell)_{\varepsilon_{i,j}\leq \ell \leq r_{i,j}}$
and
$M=\diag(m^{(i,j)}_h)_{0\leq h \leq r_{i,j}-\varepsilon_{i,j}}$.
Then $P^{(i,j)}=M^{-1} \tran{Q^{(i,j)}} K$.

\begin{theorem} 
The coherent configuration $\mathcal C_1$ is $P$-polynomial.
\end{theorem}

\begin{proof}
Set $\theta=\pi/N$, and let $\boldsymbol P^{(i,j)}_\ell$ denote the
$\ell$-th column of the normalized first eigenmatrix
$\overline{P^{(i,j)}}$. 
By \cref{rmk:criminal}, it suffices to distinguish the
cases $i\not\equiv j\pmod 2$ and $i\equiv j\pmod 2$.\\
\noindent
{\bf{Case I:}} $i\not\equiv j\pmod 2$. \\
In this case, $M=\diag(1,2,\ldots,2)$. If $N$ is even, then $K=\diag(2,\ldots,2,2)$ .
If $N$ is odd, then $K=\diag(2,\ldots,2,1)$.
Then $\varepsilon_{i,j}=1$ and $Q_h^{(i,j)}(\ell)=2\cos((2\ell-1)h\theta)$.
Moreover,
\begin{equation}\label{eqn:wait}
\overline{P^{(i,j)}}
 =
 \frac{1}{2}\Lambda^{-1}\tran{Q^{(i,j)}}K,    
\end{equation}
 where $\Lambda=\diag\left(
  1,Q_1^{(i,j)}(1),\ldots,Q_{s_{i,j}-1}^{(i,j)}(1)
 \right)$.
 Hence, for every $2\leq \ell \leq s_{i,j}-1$,
\[
 \overline{P_\ell^{(i,j)}}(h)
 =
 \frac{Q_h^{(i,j)}(\ell)}
      {Q_h^{(i,j)}(1)}
 =
 \frac{\cos((2\ell-1)h\theta)}
      {\cos(h\theta)}.
\]
By \cref{lem:curtain}, there exists a polynomial $f_{\ell-1}$ of degree
$\ell-1$ such that
\[
 \frac{\cos((2\ell-1)h\theta)}
      {\cos(h\theta)}
 =
 f_{\ell-1}
 \left(
  \frac{\cos(3h\theta)}{\cos(h\theta)}
 \right).
\]
Therefore,
\[
 \boldsymbol P_\ell^{(i,j)}
 =
 f_{\ell-1}\left(\boldsymbol P_2^{(i,j)}\right).
\]
If $N$ is odd, then $k_{s_{i,j}}=1$ and consequently
\[
 \boldsymbol P_{s_{i,j}}^{(i,j)}
 =
 \frac{1}{2}
 f_{s_{i,j}-1}
 \left(\boldsymbol P_2^{(i,j)}\right).
\]
Thus \cref{def:responsible} is satisfied in the case
$i\not\equiv j\pmod 2$.\\
\noindent
{\bf{Case II:}} $i\equiv j\pmod 2$.\\
We first consider a diagonal block $(i,i)$ and put
$d=s_{i,i}=\lfloor N/2\rfloor$.  If $N$ is even, then
$K=M=\diag(1,2,\ldots,2,1)$, whereas if $N$ is odd, then
$K=M=\diag(1,2,\ldots,2,2)$.  For $1\leq h,\ell\leq d-1$, we have
\[
 \alpha_\ell^{(i,i)}=\cos(2\ell\theta)
 \quad\text{and}\quad
 Q_h^{(i,i)}(\ell)
 =
 2T_h\left(\alpha_\ell^{(i,i)}\right)
\]
and, moreover,
\begin{equation}\label{eqn:recovery}
 \overline{P^{(i,i)}}=P^{(i,i)}
 =M^{-1}\tran{Q^{(i,i)}}K.
\end{equation}
Using $T_h(\cos(2\ell\theta))
 =
 T_\ell(\cos(2h\theta))$, for $1\leq\ell\leq d-1$, 
we obtain
\[
 \boldsymbol P_\ell^{(i,i)}
 =
 2T_\ell
 \left(
  \frac{\boldsymbol P_1^{(i,i)}}{2}
 \right).
\]
Suppose first that $N$ is even, so that $d=N/2$.  Directly from
\cref{eqn:concrete},
\[
 Q_h^{(i,i)}(d)=
 \begin{cases}
 (-1)^h,&h=0 \text{~or~} d,\\
 2(-1)^h,&1\leq h\leq d-1.\\
 \end{cases}
\]
Since $k_d^{(i,i)}=m_0^{(i,i)}=m_d^{(i,i)}=1$ and
$m_h^{(i,i)}=2$ for $1\leq h\leq d-1$, \cref{eqn:version} gives
\[
 P_d^{(i,i)}(h)=(-1)^h
 =T_d\left(\frac{P_1^{(i,i)}(h)}{2}\right).
\]
Consequently,
\[
 \boldsymbol P_d^{(i,i)}
 =
 T_d
 \left(
  \frac{\boldsymbol P_1^{(i,i)}}{2}
 \right).
\]
If $N$ is odd, then $d=(N-1)/2$ and $k_h^{(i,i)}=m_h^{(i,i)}=2$
for $1\leq h\leq d$.  Hence the same Chebyshev identity gives
\[
 \boldsymbol P_\ell^{(i,i)}
 =
 2T_\ell
 \left(
  \frac{\boldsymbol P_1^{(i,i)}}{2}
 \right),
 \qquad
 1\leq\ell\leq d.
\]
It remains to consider a same-parity off-diagonal block, namely,  $i\ne j$
and $i\equiv j\pmod 2$.  
By
\cref{eqn:junior}, the distinguished column
$\boldsymbol P_1^{(i,j)}$ is the all-ones column and
\[
 \overline{\boldsymbol P}_{\ell+1}^{(i,j)}
 =
 \boldsymbol P_\ell^{(i,i)}
 \qquad(0\leq\ell\leq d).
\]
Therefore, if $N$ is odd,
\[
 \overline{\boldsymbol P}_{\ell+1}^{(i,j)}
 =
 2T_\ell\left(
 \frac{\overline{\boldsymbol P}_2^{(i,j)}}{2}
 \right)
 \qquad(1\leq\ell\leq d),
\]
whereas if $N$ is even,
\[
 \overline{\boldsymbol P}_{\ell+1}^{(i,j)}
 =
 2T_\ell\left(
 \frac{\overline{\boldsymbol P}_2^{(i,j)}}{2}
 \right)
 \qquad(1\leq\ell\leq d-1)
\]
and
\[
 \overline{\boldsymbol P}_{d+1}^{(i,j)}
 =
 T_d\left(
 \frac{\overline{\boldsymbol P}_2^{(i,j)}}{2}
 \right).
\]
The polynomial associated with
$\overline{\boldsymbol P}_{\ell+1}^{(i,j)}$ has degree
$\ell=(\ell+1)-\varepsilon_{i,j}$.  Thus
\cref{def:responsible} is satisfied for every same-parity block.
Hence the coherent configuration is $P$-polynomial.
\end{proof}
\subsection{Coherent configuration associated to Terwilliger algebra of $H(n,2)$}
Let $n$ be a positive integer, $X=\{0,1\}^n$ and $R_i=\{(x, y) \in X \times X \mid d_H(x, y)=i\}$ for $i \in\{0,1, \ldots, n\}$ where $d_H(x, y)=\left|\left\{\ell \in\{1, \ldots, n\} \mid x_{\ell} \neq y_{\ell}\right\}\right|$ is the Hamming distance.
The pair $\left(X,\left\{R_i\right\}_{i=0}^n\right)$ is a binary Hamming association scheme denoted by $H(n,2)$. 

For $i \in\{0,1, \ldots, n\}$, denote $X_i=\{x \in X \mid d_H(x, \mathbf 0)=i\}$ where $\mathbf 0=(0, \ldots, 0)$.
For $i,j\in\{0,1,\ldots,n\}$, put $d_{i,j}=\min\{i,j,n-i,n-j\}.$
For $\varepsilon_{i,j}\leq\ell\leq
d_{i,j}+\varepsilon_{i,j},$ define
\begin{equation}\label{eqn:examination}
R_\ell^{(i,j)}=\left\{(x,y)\in X_i\times X_j\mid 
\ell=\min\{i,j\}+\varepsilon_{i,j}-|\operatorname{supp}(x)\cap\operatorname{supp}(y)|
\right\},    
\end{equation}
where $\operatorname{supp}(x)=\{\ell \mid x_\ell=1\}$.
The pair $\mathcal C_2=\left(X,\left\{R_{\ell}^{(i,j)}\right\}\right)$ is a coherent configration, whose coherent algebra is also known as the {\it Terwilliger algebra of $H(n,2)$}.
A basis of Terwilliger algebra of $H(n,2)$ was constructed as follows (see \cite{Vallentin-2009} or \cite{Schrijver-2005}). 

\begin{lemma}[{\cite[Theorem~4.1]{Vallentin-2009}}, {\cite[Theorem~5.23]{Schrijver-2005}}] \label{lem:cucumber}
For integers $h\in\{0,\ldots,\lfloor n/2\rfloor\}$ and $i,j\in\{h,\ldots,n-h\},$ set  $v(x,y)=|\mathrm{supp}(x)\setminus \mathrm{supp}(y)|,$ $W_h^{(i,j)}=\left(\frac{(-j)_h(i-n)_h}{(-i)_h(j-n)_h}\right)^{-1/2}$ and define matrices $E_{h}^{(i,j)}$ indexed by $X\times X$ whose $(x,y)$-entry is given by
\[\begin{aligned}
 E_{h}^{(i,j)}(x,y)=
 \begin{cases}
 \frac{m_h^{(i,j)} W_h^{(i,j)}}{\left(\binom{n}{i}\binom{n}{j}\right)^{1/2}} Q_h(v(x,y);-(n-i)-1,-i-1,j) & \text{ if }x\in X_i,y\in X_j,\\
0 & \text{ otherwise}. 
\end{cases} 
\end{aligned}\]
Then $\{E^{(i,j)}_{h}\}_{0\leq h\leq d_{i,j}}$ form a basis satisfying (B1)-(B4). 
In particular, the coherent configuration $\mathcal C_2$ is $Q$-polynomial.  
\end{lemma}
In the following, we will prove it is also $P$-polynomial.
Consider the second eigenmatrix $Q^{(i,j)}$ between $X_i$ and $X_j$ whose $(\ell,h)$-entry is denoted by $Q_h^{(i,j)}(\ell)$.
By \cref{eqn:election}, we have  $Q^{(i,j)}_h(\ell)=\sqrt{|X_i| |X_j|} E_{h}^{(i,j)}(x,y)$ if $(x,y)\in R_\ell^{(i,j)}$.
One can check that $m_h^{(i,j)}=\binom{n}{h}-\binom{n}{h-1}$ and $k_{\ell}^{(i,j)}=\binom{i}{\ell-\varepsilon_{i,j}} \binom{n-i}{j-i+\ell-\varepsilon_{i,j}}$.

For simplicity, denote $Q_k(u):=Q_k(u;-(n-i)-1,-i-1,j)$ and $R_k(\lambda(u)):=R_k(\lambda(u);-(n-i)-1,-i-1,j)$.
It follows from \cref{eqn:examination} that 
\begin{equation}\label{eqn:bleed}
Q^{(i,j)}_h(\ell)=m_h^{(i,j)} W_h^{(i,j)}  Q_h\left(\ell-\varepsilon_{i,j}\right).  
\end{equation}

\begin{theorem}\label{thm:coach}
The coherent configuration $\mathcal C_2$ is $P$-polynomial.  
\end{theorem}
\begin{proof}
According to \cref{prop:heart}(3) and the relation  between Hahn and dual Hahn polynomials by \cref{eqn:stress}, the $(h,\ell)$-entry of the normalized first eigenmatrix $P^{(i,j)}$ is calculated below
\[\begin{aligned}
 \overline{P_\ell^{(i,j)}}(h)&=\frac{P^{(i,j)}_{\ell}(h)}{P^{(i,j)}_{\varepsilon_{i,j}}(h)}=\frac{Q_h^{(i, j)}(\ell) k_{\ell}^{(i, j)}}{Q_h^{(i, j)}(\varepsilon_{i,j}) k_{\varepsilon_{i,j}}^{(i, j)}}=\frac{k_{\ell}^{(i, j)}  Q_h\left(\ell-\varepsilon_{i,j};n,i,j\right)}{k_{\varepsilon_{i,j}}^{(i, j)}  Q_h\left(0;n,i,j\right)} =\frac{k_{\ell}^{(i, j)}  R_{\ell-\varepsilon_{i,j}}(\lambda(h))}{k_{\varepsilon_{i,j}}^{(i, j)}}.
\end{aligned}\]
In particular, when $\ell=1+\varepsilon_{i,j}$, we have
\[\begin{aligned}
 \overline{P^{(i,j)}_{1+\varepsilon_{i,j}}}(h)
 &=\frac{k_{1+\varepsilon_{i,j}}^{(i, j)}  R_{1}(\lambda(h))}{k_{\varepsilon_{i,j}}^{(i, j)}  }.  
\end{aligned}\]
Since $R_1(\lambda(h))=1-\frac{h(n-h+1)}{j (n-i)}=1+\frac{\lambda(h)}{j (n-i)}$, we have 
\begin{equation}\label{eqn:energy}
\lambda(h)=j(n-i)(R_1(\lambda(h))-1).    
\end{equation}
From the definition of dual Hahn polynomials, we know that $R_{\ell}(\lambda(h))$ is a polynomial of degree $\ell$ in terms of $\lambda(h)$, as well as $R_1(\lambda(h))$ by \cref{eqn:energy}.
Then there exists a polynomial $g_{\ell}(u)$ of degree $\ell$ such that $R_{\ell-\varepsilon_{i,j}}(\lambda(h))=g_{\ell-\varepsilon_{i,j}}\left(R_1(\lambda(h))\right).$
Therefore, we have
\[
\frac{k_{\varepsilon_{i,j}}^{(i, j)} }{ k_{\ell}^{(i, j)}}\overline{P_\ell^{(i,j)}}(h) 
=
g_{\ell-\varepsilon_{i,j}}\!\left(
\frac{ k_{\varepsilon_{i,j}}^{(i, j)} }{ k_{1+\varepsilon_{i,j}}^{(i, j)}} \overline{P_{1+\varepsilon_{i,j}}^{(i, j)}}(h)
\right).
\]
Take a polynomial $v_\ell(u)=\frac{k_{\ell}^{(i, j)}}{k_{\varepsilon_{i,j}}^{(i, j)}}\cdot g_{\ell-\varepsilon_{i,j}}\!\left(
\frac{k_{\varepsilon_{i,j}}^{(i, j)} }{ k_{1+\varepsilon_{i,j}}^{(i, j)}} u
 \right)$, then $\overline{P_\ell^{(i,j)}}(h) =v_\ell\left(\overline{P_{1+\varepsilon_{i,j}}^{(i, j)}}(h)\right)$.
 This completes the proof.
\end{proof}

\subsection{Coherent configuration arising from all the subspaces of $\mathbb{F}_q^n$}

In this subsection, we will show that the coherent configuration of all the subspaces of $\F_q^n$ is $P$- and $Q$-polynomial. 

Let $n$ be a positive integer, $\mathbb{F}_q$ be a finite field of order $q$, $\mathcal{P}(\mathbb{F}_q^n)$, called the projective space, be the set of all the subspaces of $\mathbb{F}_q^n$, $\mathcal{G}_q(n,k)$  be the subset of $\mathcal{P}(\mathbb{F}_q^n)$ having dimension $k$. 
The general linear group $G=GL_n(\mathbb{F}_q)$ acts on $\mathcal{P}(\mathbb{F}_q^n)$, and its orbits are $\mathcal{G}_q(n,k)$with $0\leq k\leq n$.  
Moreover, the $G$-orbits acting on the pair $(x,y)$ of $\mathcal{P}(\mathbb{F}_q^n)$, also known as orbitals, are characterized by $\dim x,\dim y$, and $\dim (x\cap y)$. 
Denote $X=\mathcal{P}(\mathbb{F}_q^n)$ and $X_k=\mathcal{G}_q(n,k)$ for $k\in\{0,1,\ldots,n\}$.
For integers $i,j,\ell$ with $k\leq i\leq j\leq n-k$ and  $\varepsilon_{i,j}\leq \ell \leq \min\{i,n-j\}+\varepsilon_{i,j}$ , define relations
\begin{equation}\label{eqn:appeal}
 R_\ell^{(i,j)}=\{(x, y) \in X_i \times X_j \mid  \ell=i+\varepsilon_{i,j}- \dim (x\cap y)\}.   
\end{equation}
The pair $\mathcal C_3=\left(X,\left\{R_\ell^{(i,j)}\right\}\right)$ is a coherent configuration. 

Now we follow the notation and statement in \cite[Section 4]{BPV-2012} to introduce the primitive idempotents of the coherent configuration $\mathcal C_3$.
Let $\R^X=\{f: X\rightarrow \R\}$
be a space equipped with the inner product
\[
(f,g)=\frac{1}{|X|}\sum_{x\in X}f(x)g(x).
\]
Delsarte \cite{Delsarte-1978} proved that
\[
\mathbb{R}^X=\bigoplus_{i=0}^{n}\mathbb{R}^{X_i}=\bigoplus_{i=0}^{n}\bigoplus_{k=0}^{\min\{i,n-i\}}
H_{k,i},
\]
where each $H_{k,i}$ is an irreducible $G$-module. Moreover, $H_{k,i}\cong H_{k,k}$ for $k\le i\le n-k$.
Therefore,
\[
\mathbb{R}^X
=
\bigoplus_{k=0}^{\lfloor n/2\rfloor}
\mathcal{I}_k,
\qquad
\mathcal{I}_k=
\bigoplus_{i=k}^{n-k}
H_{k,i},
\]
where $\mathcal{I}_k$ is the isotypic component corresponding to the irreducible module $H_{k,k}$.
Let $h_k=\dim H_{k,k}={n \brack k}-{n \brack k-1}.$
For each $H_{k,i}$, choose an orthogonal basis $\{e_{k i1},e_{k i2},\ldots,e_{k ih_k}\}$ satisfying
\[
(e_{kir},e_{kis})=\delta_{r,s} {n-2k \brack i-k}q^{k(i-k)}.
\]
Since the norm depends only on $k$ and $i$, we normalize the basis by setting
\[
\phi_{kir}
=\left({n-2k \brack i-k}
q^{k(i-k)}\right) ^{-\frac{1}{2}} \cdot e_{kir},
\qquad
r=1,\ldots,h_k.
\]
Then $\{\phi_{kir}\}_{r=1}^{h_k}$ forms a standard orthonormal basis of $H_{k,i}$.
Define $E_{k}^{(i,j)}(x,y)$ by
\[
E_{k}^{(i,j)}(x,y)=\frac{1}{|X|}\sum_{r=1}^{h_k}\phi_{kir}(x)\phi_{kjr}(y),
\qquad
k\le i,j\le n-k.
\]
According to the orthonormal basis $\{\phi_{kir}\}_{r=1}^{h_k}$, the matrices $E_{k}^{(i,j)}$ satisfy
\[
E_{k}^{(i,j)}E_{h}^{(i',j')}=\delta_{k,h}\delta_{j,i'}E_{h}^{(i,j')}.
\]

For the coherent configuration from all subspaces of $\F_q^n$, we need to follow the $q$-Hahn polynomials $Q_h(n,i,j; u)$ used in \cite{BPV-2012}.
For integers $n,i,j,h$  with 
$0\leq h \leq \min\{i, n-j\}$,  the $q$-Hahn polynomials are the polynomials uniquely determined by the following properties: 
\begin{enumerate}
    \item  $Q_h(n,i,j; u)$ has degree $h$ in the variable $[u]=q^{1-u}{u \brack 1}$;
    \item They are orthogonal polynomials in terms of weights 
    \[
    w(n,i,j;u)={i \brack u} {n-i \brack j-i+u} q^{u(j-i+u)}, \quad 0\leq u\leq \min\{i,n-j\};
    \]
    \item  $Q_h(n,i,j;0)=1$.
\end{enumerate}

\begin{lemma}[{\cite[Theorem~4.2, Lemma~4.5]{BPV-2012}}]
For $h\leq i\leq j\leq n-h$,  define
\begin{align}\label{eqn:cucumber}
 E_{h}^{(i,j)}(x,y)=\begin{cases} m_h \frac{ {j-h \brack i-h} }{{n \brack j} {j \brack i}}  Q_h(n,i,j;i-\dim(x\cap y)) & \text{ if}\quad x\in X_i,\ y\in X_j,\\
0 & \text{otherwise}.   
\end{cases} 
\end{align}
Then $\left\{E^{(i,j)}_{h}\right\}_{0\leq h\leq i}$ form a basis of coherent configuration of all subspaces of $\mathcal{P}(\mathbb{F}_q^n)$ satisfying (B1)-(B4). 
In particular, the coherent configuration $\mathcal C_3$ is $Q$-polynomial.  
\end{lemma}
The polynomial $Q_h(n,i,j; u)$ and the classical dual $q$-Hahn polynomial in \cref{eqn:anniversary} differ by a nonzero scalar factor.
According to Theorem 2.5, Eq. (2.1) and Proposition 2.3 in \cite{Dunkl-1978}, the precise relation is given below
\begin{equation}\label{eqn:dash}
 Q_h(n,i,j; u)=C_h^{(i,j)} \cdot Q_h(u;q^{-(n-i)-1},q^{-i-1},j|q),    
 \end{equation} 
 where $C_h^{(i,j)} =\frac{(q^{i-j})_h q^{h n-3h(h-1)/2}(q^{j})_h}{(-1)^h q^{jh} (q^{n-j})_h (q^{i})_h}$. 
Denote $Q_h(u | q)
:=
Q_h\left(
u;q^{-(n-i)-1},q^{-i-1},j| q
\right)$ and $R_h(\mu(u)| q)
:=
R_h\left(
\mu(u);q^{-(n-i)-1},q^{-i-1},j| q
\right).$ for simplicity.

\begin{theorem}\label{thm:maze}
 The coherent configuration $\mathcal C_3$ is $P$-polynomial.     
\end{theorem}
\begin{proof}
By \cref{eqn:appeal,eqn:cucumber,eqn:dash}, we have
\begin{equation}
Q^{(i,j)}_h(\ell)=m_h^{(i,j)}C_h^{(i,j)} \sqrt{\frac{|X_i|}{|X_j|}} \frac{ {j-h \brack i-h} }{{j \brack i}} Q_h\left(\ell-\varepsilon_{i,j}\right|q).  
\end{equation}
Similarly, we obtain the normalized first eigenmatrix
\[\begin{aligned}
 \overline{P_\ell^{(i,j)}}(h)&=\frac{P^{(i,j)}_{\ell}(h)}{P^{(i,j)}_{\varepsilon_{i,j}}(h)}=\frac{Q_h^{(i, j)}(\ell) k_{\ell}^{(i, j)}}{Q_h^{(i, j)}(\varepsilon_{i,j}) k_{\varepsilon_{i,j}}^{(i, j)}}=\frac{k_{\ell}^{(i, j)}  Q_h\left(\ell-\varepsilon_{i,j}|q\right)}{k_{\varepsilon_{i,j}}^{(i, j)}  Q_h\left(0|q\right)} =\frac{k_{\ell}^{(i, j)}  R_{\ell-\varepsilon_{i,j}}(\mu(h)|q)}{k_{\varepsilon_{i,j}}^{(i, j)}}.
\end{aligned}\]
In particular, when $\ell=1+\varepsilon_{i,j}$, we have
$\overline{P^{(i,j)}_{1+\varepsilon_{i,j}}}(h)
 =\frac{k_{1+\varepsilon_{i,j}}^{(i, j)}  R_{1}(\mu(h)|q)}{k_{\varepsilon_{i,j}}^{(i, j)}  }.$
From the definition of dual $q$-Hahn polynomials, we obtain
\[
R_1(\mu(h)|q)
=
\frac{\mu(h)-\theta_{ij}}
     {(1-q^{-n+i})(1-q^{-j})},
\]
where $\theta_{ij}
=
q^{-n+i}+q^{-j}-q^{-n+i-j}+q^{-n-1}.$
Therefore, $R_{\ell}(\mu(h)|q)$ is a polynomial of degree $\ell$ in terms of $\mu(h)$, as well as $R_1(\mu(h)|q)$.
Then there exists a polynomial $g_{\ell}(u)$ of degree $\ell$ such that $R_{\ell-\varepsilon_{i,j}}(\mu(h))=g_{\ell-\varepsilon_{i,j}}\left(R_1(\mu(h))\right).$
Namely, we have
\[
\frac{k_{\varepsilon_{i,j}}^{(i, j)} }{ k_{\ell}^{(i, j)}}\overline{P_\ell^{(i,j)}}(h) 
=
g_{\ell-\varepsilon_{i,j}}\!\left(
\frac{ k_{\varepsilon_{i,j}}^{(i, j)} }{ k_{1+\varepsilon_{i,j}}^{(i, j)}} \overline{P_{1+\varepsilon_{i,j}}^{(i, j)}}(h)
\right).
\]
Take a polynomial $v_\ell(u)=\frac{k_{\ell}^{(i, j)}}{k_{\varepsilon_{i,j}}^{(i, j)}}\cdot g_{\ell-\varepsilon_{i,j}}\!\left(
\frac{k_{\varepsilon_{i,j}}^{(i, j)} }{ k_{1+\varepsilon_{i,j}}^{(i, j)}} u
 \right)$, then $\overline{P_\ell^{(i,j)}}(h) =v_\ell\left(\overline{P_{1+\varepsilon_{i,j}}^{(i, j)}}(h)\right)$.
 This completes the proof.    
\end{proof}

We conclude this section by showing that the three families considered
above satisfy the compatibility condition in \cref{prop:lump}.
Consequently, the corresponding cross-block intersection matrices
$B_1^{(i,i,j)}$ are tridiagonal for all distinct $i,j$.
\begin{proposition}
For each of the three families above, every cross-block intersection matrix $B_1^{(i,i,j)}$is tridiagonal for all $i,j$.
\end{proposition}

\begin{proof}
We verify the condition in \cref{prop:lump} for each family. 
 Assume that $i\leq j$.
\begin{enumerate}
    \item {\bf{Coherent configuration $\mathcal C_1$ arising from a tight
Euclidean $t$-design in $\mathbb R^2$.}} \\
Set $\theta=\pi/N$. If
$i\not\equiv j\pmod 2$, then 
\[\overline{P_2^{(i,j)}}(h)=\frac{\cos(3h\theta)}{\cos(h\theta)}=4\cos^2(h\theta)-3, \quad 
\overline{P_1^{(i,i)}}(h)=2\cos(2h\theta)=4\cos^2(h\theta)-2.\]
Hence
\[
\overline{P_1^{(i,i)}}(h)
=
\overline{P_2^{(i,j)}}(h)+1.
\]
If $i\equiv j\pmod 2$, then, under the natural ordering of the
relations,
\[
\overline{P_1^{(i,i)}}(h)
=
\overline{P_2^{(i,j)}}(h).
\]
\item {\bf{Coherent configuration $\mathcal C_2$ associated with the
Terwilliger algebra of $H(n,2)$.}}\\
Since $R_1^{i,j}(\lambda(h))=1+\frac{\lambda(h)}{j(n-i)}$ with $\lambda(h)=h(h-n-1)$,
we have 
\[ \overline{P_1^{(i,i)}}(h)=i(n-i)+\lambda(h), \quad \overline{P_2^{(i,j)}}(h)=\frac{k_2^{(i,j)}}{k_1^{(i,j)}}\left(1+\frac{\lambda(h)}{j(n-i)}\right).\]
Therefore,
\[
\overline{P_1^{(i,i)}}(h)
=
\alpha_{i,j}\overline{P_2^{(i,j)}}(h)
+\beta_{i,j},
\]
where $\alpha_{i,j}
=
\frac{j(n-i)k_1^{(i,j)}}{k_2^{(i,j)}} \neq 0$ and $\beta_{i,j}=(i-j)(n-i)$.

\item {\bf{Coherent configuration $\mathcal C_3$ on the set of all
subspaces of $\mathbb F_q^n$.}}\\
 Set $D_{i,j}=(1-q^{-n+i})(1-q^{-j})$
and $\vartheta_{i,j}=q^{-n+i}+q^{-j}-q^{-n+i-j}+q^{-n-1}$.
Since $R_1^{i,j}(\mu(h)\mid q)
=
\frac{\mu(h)-\vartheta_{i,j}}{D_{i,j}}$,
we obtain
\[
\overline{P_1^{(i,i)}}(h)
=
\frac{k_1^{(i,i)}}{D_{i,i}}
\bigl(\mu(h)-\vartheta_{i,i}\bigr), \quad 
\overline{P_2^{(i,j)}}(h)
=
\frac{k_2^{(i,j)}}{k_1^{(i,j)}D_{i,j}}
\bigl(\mu(h)-\vartheta_{i,j}\bigr).
\]
Consequently,
\[
\overline{P_1^{(i,i)}}(h)
=
\alpha_{i,j}\overline{P_2^{(i,j)}}(h)
+\beta_{i,j},
\]
where $\alpha_{i,j}=\frac{k_1^{(i,i)}k_1^{(i,j)}D_{i,j}}{k_2^{(i,j)}D_{i,i}} \neq 0$ and 
$\beta_{i,j}=\frac{k_1^{(i,i)}\bigl(\vartheta_{i,j}-\vartheta_{i,i}\bigr)}{D_{i,i}}$.
\end{enumerate}
According to \cref{prop:lump}, every corresponding cross-block intersection
matrix $B_1^{(i,i,j)}$ is tridiagonal for all three families.
\end{proof}

\section{Concluding remarks and open problems}

In this paper, we introduced a new notion of $P$-polynomial coherent
configurations. We also presented examples that are $P$-polynomial in
the sense of \cref{def:responsible}, but not in the sense of
\cref{def:formal}. This leads to the following question.

\begin{problem}
Are there any more $P$-polynomial coherent configurations and $Q$-polynomial coherent configurations?
\end{problem}
Next step is to develop a theory of duality for coherent
configurations. In particular, we propose the following questions.

\begin{problem}
\begin{enumerate}
    \item Is there a notion of a ``dual'' coherent configuration under
    suitable algebraic conditions on the vertex set?
    \item Can one define self-dual $P$- and $Q$-polynomial coherent
    configurations?
    \item Can one define formally self-dual $P$- and $Q$-polynomial
    coherent configurations?
\end{enumerate}
\end{problem}

For $P$- and $Q$-polynomial association schemes, Delsarte theory plays a fundamental role in proofs of Erd\H{o}s--Ko--Rado theorem \cite{MR771733,MR867648,MR657052,MR887358,MR3063159}.
There has also been an approach to the Erd\H{o}s--Ko--Rado theorem in the setting of coherent configurations \cite{MR3194752}. 
This raises the following problem.

\begin{problem}
Can Delsarte's approach to the Erd\H{o}s--Ko--Rado theorem be extended to the framework of $P$- and $Q$-polynomial coherent configurations?
\end{problem}

\section*{Acknowledgments}
     Sho Suda is supported by JSPS KAKENHI Grant Numbers 26K06904. 
     Yan Zhu is supported by National Natural Science Foundation of China No.\ 12571353 and Natural Science Foundation of Shanghai No.\ 24ZR1455100.

  \bibliographystyle{plain}
  \bibliography{references.bib}

@article {Suda-2022,
  author  = {Sho Suda},
  title   = {${Q}$-polynomial coherent configurations},
  journal = {Linear Algebra Appl.},
  volume  = {643},
  pages   = {166--195},
  year    = {2022},
  doi     = {10.1016/j.laa.2022.02.009},
}

@article {MR3194752,
    AUTHOR = {Suda, Sho and Tanaka, Hajime},
     TITLE = {A cross-intersection theorem for vector spaces based on
              semidefinite programming},
   JOURNAL = {Bull. Lond. Math. Soc.},
  FJOURNAL = {Bulletin of the London Mathematical Society},
    VOLUME = {46},
      YEAR = {2014},
    NUMBER = {2},
     PAGES = {342--348},
      ISSN = {0024-6093,1469-2120},
   MRCLASS = {05D05 (90C22 90C27)},
  MRNUMBER = {3194752},
MRREVIEWER = {Tanbir\ Ahmed},
       DOI = {10.1112/blms/bdt101},
       URL = {https://doi.org/10.1112/blms/bdt101},
}

@article{gavrilyuk2025extremal,
  title={Extremal orthogonal arrays},
  author={Gavrilyuk, Alexander L and Suda, Sho},
  journal={arXiv preprint arXiv:2512.23459},
  year={2025}
}

@article {MR2677683,
    AUTHOR = {Suda, Sho},
     TITLE = {Coherent configurations and triply regular association schemes
              obtained from spherical designs},
   JOURNAL = {J. Combin. Theory Ser. A},
  FJOURNAL = {Journal of Combinatorial Theory. Series A},
    VOLUME = {117},
      YEAR = {2010},
    NUMBER = {8},
     PAGES = {1178--1194},
      ISSN = {0097-3165,1096-0899},
   MRCLASS = {05E30 (05B30)},
  MRNUMBER = {2677683},
MRREVIEWER = {Paul-Hermann\ Zieschang},
       DOI = {10.1016/j.jcta.2010.03.016},
       URL = {https://doi.org/10.1016/j.jcta.2010.03.016},
}

@article {MR771733,
    AUTHOR = {Wilson, Richard M.},
     TITLE = {The exact bound in the {E}rd{\H{o}}s-{K}o-{R}ado theorem},
   JOURNAL = {Combinatorica},
  FJOURNAL = {Combinatorica. An International Journal of the J\'anos Bolyai
              Mathematical Society},
    VOLUME = {4},
      YEAR = {1984},
    NUMBER = {2-3},
     PAGES = {247--257},
      ISSN = {0209-9683},
   MRCLASS = {05A05 (05C35 05C65)},
  MRNUMBER = {771733},
MRREVIEWER = {Noga\ Alon},
       DOI = {10.1007/BF02579226},
       URL = {https://doi.org/10.1007/BF02579226},
}

@article {MR867648,
    AUTHOR = {Frankl, P. and Wilson, R. M.},
     TITLE = {The {E}rd{\H{o}}s-{K}o-{R}ado theorem for vector spaces},
   JOURNAL = {J. Combin. Theory Ser. A},
  FJOURNAL = {Journal of Combinatorial Theory. Series A},
    VOLUME = {43},
      YEAR = {1986},
    NUMBER = {2},
     PAGES = {228--236},
      ISSN = {0097-3165,1096-0899},
   MRCLASS = {05A05},
  MRNUMBER = {867648},
MRREVIEWER = {G.\ F.\ Clements},
       DOI = {10.1016/0097-3165(86)90063-4},
       URL = {https://doi.org/10.1016/0097-3165(86)90063-4},
}

@article {MR657052,
    AUTHOR = {Moon, Aeryung},
     TITLE = {An analogue of the {E}rd{\H{o}}s-{K}o-{R}ado theorem for the
              {H}amming schemes {$H(n,\,q)$}},
   JOURNAL = {J. Combin. Theory Ser. A},
  FJOURNAL = {Journal of Combinatorial Theory. Series A},
    VOLUME = {32},
      YEAR = {1982},
    NUMBER = {3},
     PAGES = {386--390},
      ISSN = {0097-3165,1096-0899},
   MRCLASS = {05A05},
  MRNUMBER = {657052},
MRREVIEWER = {Zolt\'an\ F\"uredi},
       DOI = {10.1016/0097-3165(82)90054-1},
       URL = {https://doi.org/10.1016/0097-3165(82)90054-1},
}

@article {MR887358,
    AUTHOR = {Huang, Ta Yuan},
     TITLE = {An analogue of the {E}rd{\H{o}}s-{K}o-{R}ado theorem for the
              distance-regular graphs of bilinear forms},
   JOURNAL = {Discrete Math.},
  FJOURNAL = {Discrete Mathematics},
    VOLUME = {64},
      YEAR = {1987},
    NUMBER = {2-3},
     PAGES = {191--198},
      ISSN = {0012-365X,1872-681X},
   MRCLASS = {05C30 (05A05 94B25)},
  MRNUMBER = {887358},
       DOI = {10.1016/0012-365X(87)90188-9},
       URL = {https://doi.org/10.1016/0012-365X(87)90188-9},
}

@article {MR3063159,
    AUTHOR = {Tanaka, Hajime},
     TITLE = {The {E}rd{\H{o}}s-{K}o-{R}ado theorem for twisted {G}rassmann
              graphs},
   JOURNAL = {Combinatorica},
  FJOURNAL = {Combinatorica. An International Journal on Combinatorics and
              the Theory of Computing},
    VOLUME = {32},
      YEAR = {2012},
    NUMBER = {6},
     PAGES = {735--740},
      ISSN = {0209-9683,1439-6912},
   MRCLASS = {05E30 (05D05)},
  MRNUMBER = {3063159},
MRREVIEWER = {Alexander\ L.\ Gavrilyuk},
       DOI = {10.1007/s00493-012-2798-5},
       URL = {https://doi.org/10.1007/s00493-012-2798-5},
}

@article{Bajnok-2006,
    AUTHOR = {Bajnok, B\'ela},
     TITLE = {On {E}uclidean designs},
   JOURNAL = {Adv. Geom.},
  FJOURNAL = {Advances in Geometry},
    VOLUME = {6},
      YEAR = {2006},
    NUMBER = {3},
     PAGES = {423--438},
      ISSN = {1615-715X,1615-7168},
   MRCLASS = {05B30 (33C45 41A55 51E05)},
  MRNUMBER = {2248260},
MRREVIEWER = {Ian\ Blake},
       DOI = {10.1515/ADVGEOM.2006.026},
       URL = {https://doi.org/10.1515/ADVGEOM.2006.026},
}

@article{Dunkl-1978,
    AUTHOR = {Dunkl, Charles F.},
     TITLE = {An addition theorem for some {$q$}-{H}ahn polynomials},
   JOURNAL = {Monatsh. Math.},
  FJOURNAL = {Monatshefte f\"ur Mathematik},
    VOLUME = {85},
      YEAR = {1978},
    NUMBER = {1},
     PAGES = {5--37},
      ISSN = {0026-9255,1436-5081},
   MRCLASS = {33A70},
  MRNUMBER = {486703},
MRREVIEWER = {W.\ A.\ Al-Salam},
       DOI = {10.1007/BF01300958},
       URL = {https://doi.org/10.1007/BF01300958},
}

@article{BPV-2012,
    AUTHOR = {Bachoc, Christine and Passuello, Alberto and Vallentin, Frank},
     TITLE = {Bounds for projective codes from semidefinite programming},
   JOURNAL = {Adv. Math. Commun.},
  FJOURNAL = {Advances in Mathematics of Communications},
    VOLUME = {7},
      YEAR = {2013},
    NUMBER = {2},
     PAGES = {127--145},
      ISSN = {1930-5346,1930-5338},
   MRCLASS = {94B65 (90C22)},
  MRNUMBER = {3063504},
MRREVIEWER = {Vladimir\ D.\ Tonchev},
       DOI = {10.3934/amc.2013.7.127},
       URL = {https://doi.org/10.3934/amc.2013.7.127},
}

@article{LSK-2010,
    AUTHOR = {Koekoek, Roelof and Lesky, Peter A. and Swarttouw, Ren\'e{}
              F.},
     TITLE = {Hypergeometric orthogonal polynomials and their
              {$q$}-analogues},
    SERIES = {Springer Monographs in Mathematics},
      NOTE = {With a foreword by Tom H. Koornwinder},
 PUBLISHER = {Springer-Verlag, Berlin},
      YEAR = {2010},
     PAGES = {xx+578},
      ISBN = {978-3-642-05013-8},
   MRCLASS = {33C45 (33-02 33D45)},
  MRNUMBER = {2656096},
MRREVIEWER = {Jeremy\ Lovejoy},
       DOI = {10.1007/978-3-642-05014-5},
       URL = {https://doi.org/10.1007/978-3-642-05014-5},
}

@article{Schrijver-2005,
    AUTHOR = {Schrijver, Alexander},
     TITLE = {New code upper bounds from the {T}erwilliger algebra and
              semidefinite programming},
   JOURNAL = {IEEE Trans. Inform. Theory},
  FJOURNAL = {Institute of Electrical and Electronics Engineers.
              Transactions on Information Theory},
    VOLUME = {51},
      YEAR = {2005},
    NUMBER = {8},
     PAGES = {2859--2866},
      ISSN = {0018-9448,1557-9654},
   MRCLASS = {94A29 (05E30 16S99 90C22 94B05)},
  MRNUMBER = {2236252},
       DOI = {10.1109/TIT.2005.851748},
       URL = {https://doi.org/10.1109/TIT.2005.851748},
}

@article {Vallentin-2009,
    AUTHOR = {Vallentin, Frank},
     TITLE = {Symmetry in semidefinite programs},
   JOURNAL = {Linear Algebra Appl.},
  FJOURNAL = {Linear Algebra and its Applications},
    VOLUME = {430},
      YEAR = {2009},
    NUMBER = {1},
     PAGES = {360--369},
      ISSN = {0024-3795,1873-1856},
   MRCLASS = {90C22 (33C90)},
  MRNUMBER = {2460523},
MRREVIEWER = {Yan\ Gao},
       DOI = {10.1016/j.laa.2008.07.025},
       URL = {https://doi.org/10.1016/j.laa.2008.07.025},
}

@article {Higman-1975,
  author    = {Donald G. Higman},
  title     = {Coherent Configurations. {I}. Ordinary Representation Theory},
  journal   = {Geometriae Dedicata},
  volume    = {4},
  number    = {1},
  pages     = {1--32},
  year      = {1975},
  doi       = {10.1007/BF00147398}
}

@article{HIGMAN1988411,
title = {Strongly Regular Designs and Coherent Configurations of Type $\begin{bmatrix}3& 2 \\ & 3 \end{bmatrix}$},
journal = {European Journal of Combinatorics},
volume = {9},
number = {4},
pages = {411-422},
year = {1988},
issn = {0195-6698},
doi = {https://doi.org/10.1016/S0195-6698(88)80072-6},
url = {https://www.sciencedirect.com/science/article/pii/S0195669888800726},
author = {D.G. Higman}
}

@article {MR1345694,
    AUTHOR = {Higman, D. G.},
     TITLE = {Strongly regular designs of the second kind},
   JOURNAL = {European J. Combin.},
  FJOURNAL = {European Journal of Combinatorics},
    VOLUME = {16},
      YEAR = {1995},
    NUMBER = {5},
     PAGES = {479--490},
      ISSN = {0195-6698,1095-9971},
   MRCLASS = {05B30},
  MRNUMBER = {1345694},
MRREVIEWER = {Esther\ R.\ Lamken},
       DOI = {10.1016/0195-6698(95)90003-9},
       URL = {https://doi.org/10.1016/0195-6698(95)90003-9},
}

@article{Delsarte-1978,
  author  = {Philippe Delsarte},
  title   = {Hahn Polynomials, Discrete Harmonics, and {$t$}-Designs},
  journal = {SIAM Journal on Applied Mathematics},
  volume  = {34},
  number  = {1},
  pages   = {157--166},
  year    = {1978},
  doi     = {10.1137/0134015}
}

@article{Lato-2025,
  title={P-polynomial and bipartite coherent configurations},
  author={Lato, Sabrina},
  journal={Linear Algebra and its Applications},
  volume={708},
  pages={12--41},
  year={2025},
  publisher={Elsevier}
}

@article{FILM-2025,
  title={New Constructions of Distance-Biregular Graphs},
  author={Fern{\'a}ndez, Blas and Ihringer, Ferdinand and Lato, Sabrina and Munemasa, Akihiro},
  journal={arXiv preprint arXiv:2504.21488},
  year={2025}
}

@misc{WL-2018,
  author = {Boris Yu. Weisfeiler and Andrei A. Leman},
  title  = {The Reduction of a Graph to Canonical Form and the Algebra Which Appears Therein},
  howpublished = {English translation by Grigory Ryabov},
  year         = {2018},
  note         = {Originally published in Nauchno-Technicheskaya Informatsia, Seriya~2, No.~9, pp.~12--16, 1968},
  url          = {https://www.iti.zcu.cz/wl2018/pdf/wl_paper_translation.pdf}
}

@article{JZW-2026,
  title={Euclidean designs obtained from Q-polynomial coherent configurations},
  author={Jiang, Yuchen and Zhu, Yan and Wu, Chengju},
  journal={Journal of Algebraic Combinatorics},
  volume={64},
  number={1},
  pages={10},
  year={2026},
  publisher={Springer}
}

@article{BBHS-2010,
  author    = {Eiichi Bannai and Etsuko Bannai and Masatake Hirao and Masanori Sawa},
  title     = {Cubature formulas in numerical analysis and Euclidean tight designs},
  journal   = {European Journal of Combinatorics},
  volume    = {31},
  number    = {2},
  pages     = {423--441},
  year      = {2010},
  doi       = {10.1016/j.ejc.2009.03.035},
}

@article{GS-1970,
  author  = {Goethals, J. M. and Seidel, J. J.},
  title   = {Strongly Regular Graphs Derived from Combinatorial Designs},
  journal = {Canadian Journal of Mathematics},
  volume  = {22},
  year    = {1970},
  pages   = {597--614}
}

@article{Lisonek-1997,
  author  = {Petr Lison{\v{e}}k},
  title   = {New Maximal Two-Distance Sets},
  journal = {Journal of Combinatorial Theory, Series A},
  volume  = {77},
  number  = {2},
  pages   = {318--338},
  year    = {1997},
  doi     = {10.1006/jcta.1997.2749}
}

@incollection{BB-2010,
  author    = {Eiichi Bannai and Etsuko Bannai},
  title     = {Euclidean Designs and Coherent Configurations},
  booktitle = {Algebraic Combinatorics},
  series    = {Contemporary Mathematics},
  volume    = {531},
  pages     = {59--93},
  year      = {2010},
  publisher = {American Mathematical Society},
  address   = {Providence, RI},
  doi       = {10.1090/conm/531/10468}
}

@article{BBTZ-2022,
  author  = {Eiichi Bannai and Etsuko Bannai and Hajime Tanaka and Yan Zhu},
  title   = {Tight Relative {$t$}-Designs on Two Shells in Hypercubes, and Hahn and Hermite Polynomials},
  journal = {Ars Mathematica Contemporanea},
  volume  = {22},
  number  = {2},
  pages   = {P2.01},
  year    = {2022},
  doi     = {10.26493/1855-3974.2352.eaf}
}
\end{document}